\documentclass[11pt,reqno]{amsart}

\usepackage[margin=1.15in]{geometry}

\usepackage{amsmath,amssymb,amsthm,mathtools}
\usepackage{mathrsfs}
\usepackage{bm}

\usepackage{graphicx}
\usepackage{multirow}
\usepackage{placeins}
\usepackage{float}

\usepackage{enumitem}

\usepackage{xcolor}
\usepackage{tikz}
\usepackage[colorlinks=true,
  linkcolor=blue,
  citecolor=blue,
  urlcolor=blue]{hyperref}
\usepackage[capitalize,nameinlink]{cleveref}

\theoremstyle{plain}
\newtheorem{theorem}{Theorem}[section]
\newtheorem{proposition}[theorem]{Proposition}
\newtheorem{lemma}[theorem]{Lemma}

\newtheorem{conjecture}[theorem]{Conjecture}

\theoremstyle{definition}
\newtheorem{definition}[theorem]{Definition}

\theoremstyle{remark}
\newtheorem{remark}[theorem]{Remark}

\numberwithin{equation}{section}

\makeatletter

\let\@preprintnumber\@empty
\newcommand{\preprintnumber}[1]{%
  \gdef\@preprintnumber{#1}%
}

\def\@settitle{%
  \ifx\@preprintnumber\@empty
  \else
    \vspace*{-2.6cm}
    \begin{flushright}%
      \normalfont\normalsize
      \@preprintnumber
    \end{flushright}%
    \vspace{1.6cm}
  \fi
  \begin{center}%
    \baselineskip14\p@\relax
    \bfseries\Large
    \@title
  \end{center}%
}
\def\@setauthors{%
  \begingroup
  \trivlist
  \centering\normalsize \@topsep30\p@\relax
  \advance\@topsep by -\baselineskip
  \item\relax
  \andify\authors
  \def\\{\protect\linebreak}%
  \authors
  \endtrivlist
  \endgroup
}
\makeatother

\makeatletter
\def\section{\@startsection{section}{1}%
  \z@{.7\linespacing\@plus\linespacing}{.5\linespacing}%
  {\normalfont\bfseries\centering}}

\def\subsection{\@startsection{subsection}{2}%
  \z@{.5\linespacing\@plus.7\linespacing}{.3\linespacing}%
  {\normalfont\bfseries}}

\def\subsubsection{\@startsection{subsubsection}{3}%
  \z@{.5\linespacing\@plus.7\linespacing}{.3\linespacing}%
  {\normalfont\bfseries}}
\makeatother

\definecolor{lime}{HTML}{A6CE39}
\DeclareRobustCommand{\orcidicon}{%
    \begin{tikzpicture}
    \draw[lime, fill=lime] (0,0)
    circle [radius=0.16]
    node[white] {{\fontfamily{qag}\selectfont \tiny ID}};
    \draw[white, fill=white] (-0.0625,0.095)
    circle [radius=0.007];
    \end{tikzpicture}
    \hspace{-3mm}
}
\foreach \x in {A,...,Z}{%
  \expandafter\xdef\csname orcid\x\endcsname{%
    \noexpand\href{https://orcid.org/\csname orcidauthor\x\endcsname}
    {\noexpand\orcidicon}}%
}

\usepackage[numbers,sort&compress]{natbib}

\begin{document}


\title[Arithmetic Symmetry in Ideal PTE Solutions]{
Arithmetic Symmetry in Ideal Prouhet--Tarry--Escott Solutions
}

\preprintnumber{TU-1310}

\author[Y.-D. Tsai]{
Yu-Dai Tsai\hspace{-0.5mm}\orcidA{}
}

\author[J. Lee]{
Junseok Lee\hspace{-0.5mm}\orcidB{}
}

\author[F. Takahashi]{
Fuminobu Takahashi\hspace{-0.5mm}\orcidC{}
}


\begin{abstract}
\normalsize
Motivated in part by anomaly cancellation for integral charge spectra in
chiral gauge theory, we study the symmetric locus in the ideal degree-three
Prouhet--Tarry--Escott problem. A symmetric integer solution is one whose
entries are paired about a common center $c\in \frac12\mathbb Z$. This
symmetry reduces the problem to a sum-of-two-squares equation,
$x^2+y^2=u^2+v^2$, in integer variables, subject to the appropriate parity
conditions. Thus the problem is governed by representations as sums of two
squares.
For the full symmetric locus, let $N_{\mathrm{sym}}(H)$ denote the number of
nontrivial symmetric integer solutions of height at most $H$, counted with
unordered multiset conventions and summed over the admissible centers. Then
\begin{align*}
N_{\mathrm{sym}}(H)
=
\frac{4\log 2}{3\pi^2}H^3\log H+O(H^3).
\end{align*}
The logarithmic enhancement comes from the second moment of the
sum-of-two-squares representation function. In particular, the symmetric locus
is larger than one would expect from the naive $H^3$ degree-weighted
box-counting scale alone. This asymptotic identifies a large arithmetically
structured subfamily of the ideal degree-three solution space, and suggests
that paired anomaly-free integral charge spectra reflect a fundamental
number-theoretic structure.
\end{abstract}

\maketitle

\vspace{0.5cm}
\tableofcontents

\section{Introduction}

The Prouhet--Tarry--Escott problem~\cite{Wright:1959Prouhet,BorweinIngalls1994,Borwein2002,shuwen2025survey} asks for two multisets of integers whose power sums agree up to a prescribed degree. A further motivation for studying this problem comes from anomaly cancellation in chiral gauge theory. In gauge theories with integral charge assignments, anomaly-cancellation conditions often impose Diophantine constraints on the charges, including linear and cubic relations in abelian settings; see, for example,~\cite{Bouchiat:1972iq,Batra:2005rh,deGouvea:2015pea,Nakayama:2011dj,Nakayama:2018yvj,Lohitsiri:2019fuu,Costa:2019zzy,Costa:2020dph}. Number-theoretic structures closely related to the Prouhet--Tarry--Escott problem, including equalities among linear, quadratic, and cubic sums, appear in~\cite{Lee:2026djo}.
Following~\cite{Lee:2026djo}, we refer to the corresponding anomaly-free charge spectra as Lee--Takahashi--Tsai (LTT) particle spectra\footnote{Here, the ``spectra'' denotes the collection of charged particle states in the Lee--Takahashi--Tsai construction. In chiral gauge theories, anomaly cancellation generally requires the contributions of multiple charged particles to cancel among themselves, a condition especially relevant in models with millicharged particles~\cite{Lee:2026djo}.}. 
Thus the degree-three Prouhet--Tarry--Escott system arises naturally both as a Diophantine problem and as an arithmetic model for a class of anomaly-free charge spectra.

More precisely, for integers $k\geq 1$ and $n\geq 1$, one seeks multisets
\begin{align}
A=\{a_1,\ldots,a_n\}, \qquad B=\{b_1,\ldots,b_n\}
\end{align}
such that
\begin{align}
\sum_{i=1}^n a_i^\ell=\sum_{i=1}^n b_i^\ell
\qquad \text{for } \ell=1,\ldots,k.
\end{align}
The case $n=k+1$, known as the ideal case, is the minimal nontrivial setting and is correspondingly rigid. Previous work on the problem has emphasized existence, explicit constructions, computational searches, and classification of solutions~\cite{2023arXiv230411254C}. In this paper, we study instead the height distribution of ideal solutions and ask which arithmetic structures persist at scale.

We focus on ideal solutions of degree three, so that $n=4$. Thus we study pairs of four-element multisets $A$ and $B$ satisfying equality of the first, second, and third power sums. The height of a solution is
\begin{align}
\max_i\{|a_i|,|b_i|\},
\end{align}
and we investigate the distribution of solutions of height at most $H$ as $H\to\infty$.

We use two counting functions throughout the paper.  The letter $C$ is reserved for centered counts, where the common center has been subtracted and only the intrinsic centered configuration is counted.
The letter $N$ denotes the corresponding shifted count, obtained by summing centered configurations over all admissible integer or half-integer centers subject to the height bound.  
Thus $C_{\mathrm{sym}}$ counts centered symmetric configurations and $N_{\mathrm{sym}}$ counts their shifted symmetric realizations. We count shifted solutions separately because, in physical applications such as anomaly cancellation, translations and rescalings of a charge spectrum generally define physically distinct spectra, and therefore correspond to different model setups.

The main object of this work is the symmetric locus.  A symmetric integer
solution is one whose entries are paired about a common center
$c\in\frac12\mathbb Z$.  Thus there are two center classes:
\begin{align}
c\in\mathbb Z
\qquad\text{and}\qquad
c\in\mathbb Z+\frac12.
\end{align}
After subtracting $c$ and multiplying centered coordinates by $2$, both
classes have the form
\begin{align}
2(A-c)=\{\pm x,\pm y\},\qquad
2(B-c)=\{\pm u,\pm v\},
\end{align}
where $x,y,u,v\in\mathbb Z$ have the same parity.  The case
$c\in\mathbb Z$ corresponds to even doubled coordinates, while the case
$c\in\mathbb Z+\frac12$ corresponds to odd doubled coordinates.

For such configurations, the first and third power sums vanish after
centering, and the degree-three Prouhet--Tarry--Escott conditions reduce to
the single quadratic relation
\begin{align}
x^2+y^2=u^2+v^2.
\end{align}
Thus the symmetric locus is governed by representations of integers as sums
of two squares, together with a parity condition.  This reduction is the basic
structural observation of the paper.

Our first result gives a simple infinite family of nontrivial symmetric
solutions. For integers $p>r\geq 2$, the identity
\begin{align}
(pr-1)^2+(p+r)^2=(pr+1)^2+(p-r)^2
\end{align}
produces the symmetric Prouhet--Tarry--Escott solution
\begin{align}
A=\{\pm(pr-1),\pm(p+r)\}, \qquad
B=\{\pm(pr+1),\pm(p-r)\}.
\end{align}
Counting pairs $p>r\geq 2$ with $pr\leq H-1$ gives a lower bound of order
$H\log H$ for this explicitly constructed family. In particular, symmetric
solutions are not isolated examples, but arise from a simple multiplicative
mechanism.

The main counting results concern the full symmetric locus, including both
center classes.  Let $C_{\mathrm{sym}}(H)$ denote the number of nontrivial
centered symmetric configurations of centered height at most $H$, counted
with the unordered multiset conventions fixed below.  We prove
\begin{align}
C_{\mathrm{sym}}(H)
=
\frac{2\log 2}{\pi^2}H^2\log H+O(H^2).
\end{align}
The logarithmic factor reflects the second moment of the sum-of-two-squares
representation function, while the leading constant records both the geometry
of the height box and the two parity classes.  After summing over admissible
centers, we obtain
\begin{align}
N_{\mathrm{sym}}(H)
=
\frac{4\log 2}{3\pi^2}H^3\log H+O(H^3),
\end{align}
where distinct admissible centers are counted separately.  Thus the center
parameter contributes one additional power of $H$, while preserving the
logarithmic enhancement already present in the centered problem. We note that this logarithmic enhancement goes beyond the naive heuristic of degree-weighted counting.

This arithmetic picture has a natural physical counterpart. Under the degree-three Prouhet--Tarry--Escott correspondence, symmetric solutions correspond to LTT particle spectra that naturally organize into doublet-like pairs, with states sharing nearby masses due to the charge assignments~\cite{Lee:2026djo}. The arithmetic abundance of such solutions therefore suggests a possible organizing principle for anomaly-free integral charge spectra in chiral gauge theories.

The preceding estimates motivate a comparison between the symmetric locus and the full ideal degree-three solution space. General solutions are defined by three coupled polynomial constraints, whereas the symmetric locus reduces after centering to a single quadratic equation governed by sums of two squares. In Section~\ref{sec:conjecture}, we formulate a polynomial-scale density conjecture asserting that the full counting function should satisfy
\begin{align}
N(H)=H^{3+o(1)}.
\end{align}
Under this conjecture, the symmetric locus has the same height exponent as the complete ideal degree-three solution space, 
even though its limiting density need not be positive.

The paper is organized as follows. In Section~\ref{sec:preliminaries}, we fix notation and counting conventions and prove the reduction of the symmetric locus to sums of two squares. Section~\ref{sec:family} gives an explicit infinite family and an elementary $H\log H$ lower bound. Section~\ref{sec:counting} counts the centered symmetric locus.  The two
center classes $c\in\mathbb Z$ and
$c\in\mathbb Z+\frac12$ contribute equal leading terms, and together give
\begin{align}
C_{\mathrm{sym}}(H)
=
\frac{2\log 2}{\pi^2}H^2\log H+O(H^2).
\end{align}
Section~\ref{sec:law} sums over admissible centers and proves
\begin{align}
N_{\mathrm{sym}}(H)
=
\frac{4\log 2}{3\pi^2}H^3\log H+O(H^3).
\end{align}
Section~\ref{sec:conjecture} presents finite-height evidence, formulates the polynomial-scale density conjecture, and explains why the abundance of symmetric integer solutions is not predicted by the naive degree-weighted box-counting scale. Appendix~\ref{sec:vmvt-full-pte} relates the full ideal degree-three Prouhet--Tarry--Escott solution space to the Vinogradov system $J_{4,3}(X)$, explains the raw VMVT bound
\begin{align}
N(H)\ll_{\varepsilon} H^{4+\varepsilon},
\end{align}
and outlines the affine-reduced estimates that would be needed to compare the full solution space with the symmetric count. Appendix~\ref{app:finite-enumeration} describes the finite-height enumeration
used to support the conjectural discussion.
Appendix~\ref{app:finite-sqSum-enumeration} records a companion centered
enumeration under a finite-squared-sum, or radial, cutoff.  This alternative
count is not used in the proofs of the main height-box asymptotics, but it
provides a useful comparison illustrating how the leading constant changes
when the cutoff geometry is changed.

\section{Preliminaries and the Symmetric Locus}
\label{sec:preliminaries}

We begin by fixing notation and counting conventions.  Throughout, multisets
are counted with multiplicity.  Thus $A=B$ means equality as multisets:
after ignoring order, every integer occurs with the same multiplicity in
$A$ and $B$.

\subsection{Prouhet--Tarry--Escott Solutions}

Let $k\geq 1$ and $n\geq 1$.  Two multisets of integers
\begin{align}
A=\{a_1,\ldots,a_n\}, \qquad B=\{b_1,\ldots,b_n\}
\end{align}
form a Prouhet--Tarry--Escott solution of degree $k$ if
\begin{align}
\sum_{i=1}^n a_i^\ell=\sum_{i=1}^n b_i^\ell
\qquad
\text{for } \ell=1,\ldots,k .
\end{align}
We shall write this briefly as $A\equiv_k B$.  The solution is called
trivial if $A=B$ as multisets.

In this paper we consider the ideal degree-three case.  Thus $k=3$ and
$n=4$.  A solution is therefore a pair of four-element multisets
\begin{align}
A=\{a_1,a_2,a_3,a_4\}, \qquad
B=\{b_1,b_2,b_3,b_4\}
\end{align}
satisfying
\begin{align}
\sum_{i=1}^4 a_i=\sum_{i=1}^4 b_i,\qquad
\sum_{i=1}^4 a_i^2=\sum_{i=1}^4 b_i^2,\qquad
\sum_{i=1}^4 a_i^3=\sum_{i=1}^4 b_i^3.
\end{align}

\subsection{Height and Affine Transformations}

The height of a solution $(A,B)$ is defined by
\begin{align}
\operatorname{ht}(A,B)
=
\max\{|a|,\ |b|: a\in A,\ b\in B\}.
\end{align}
We count solutions satisfying $\operatorname{ht}(A,B)\le H$.

The Prouhet--Tarry--Escott equalities are preserved under affine transformations.
Indeed, if $A\equiv_k B$, then for any 
rational numbers $\xi,\eta$ with
$\xi\neq0$, the multisets
\begin{align}
\xi A+\eta
:=
\{\xi a_1+\eta,\ldots,\xi a_n+\eta\},
\qquad
\xi B+\eta
:=
\{\xi b_1+\eta,\ldots,\xi b_n+\eta\}
\end{align}
again satisfies the degree-$k$ Prouhet--Tarry--Escott equalities. When all the elements are integers, they give another integer solution. 
This follows by expanding
$(\xi t+\eta)^\ell$ as a polynomial in $t$ of degree at most
$\ell$.

In the present work, scaling is not quotiented out in the height counts unless
explicitly stated.  Translation, however, will be used to organize symmetric
solutions by their centers.

\subsection{Centers and Shifted Solutions}

For an ideal degree-three solution $(A,B)$, the equality of first power sums
implies that the two multisets have the same mean,
\begin{align}
\frac14\sum_{a\in A}a
=
\frac14\sum_{b\in B}b.
\end{align}
Subtracting this common mean translates any solution to a centered rational
solution.  This translation need not preserve integrality.

For symmetric integer solutions, the relevant center lies in
$\frac12\mathbb Z$.  Indeed, if two integers are paired by reflection about
a center $c$, then their sum is $2c$, which must be an integer.  Thus
$c\in \frac12\mathbb Z$.  Consequently, the symmetric locus naturally
decomposes into two center classes:
\begin{align}
\mathbb Z
\qquad\text{and}\qquad
\mathbb Z+\frac12.
\end{align}
Translating all entries of a symmetric solution by an integer shifts its
center by the same integer and preserves the center class.
Distinct admissible centers will be counted
separately in the shifted counting function introduced in
Section~\ref{sec:law}.

\subsection{The Symmetric Locus}

We now isolate the main structural subclass studied in this paper.

\begin{definition}
\label{def:symmetric}
An ideal degree-three Prouhet--Tarry--Escott solution $(A,B)$ is called
symmetric if there exists $c\in \frac12\mathbb Z$ such that both centered
multisets $A-c$ and $B-c$ are invariant under sign reversal.
Equivalently, a symmetric solution is one whose entries are paired by
reflection about a common center. 

The two possibilities $c\in\mathbb Z$ and
$c\in\mathbb Z+\frac12$ will be called the integer-centered and
half-integer-centered parity classes, respectively.
\end{definition}

\begin{remark}
The half-integer-centered class is present. For example,
\begin{align}
A=\{0,4,7,11\},\qquad B=\{1,2,9,10\}
\end{align}
is symmetric about $c=11/2$.  After subtracting $11/2$ and multiplying by
$2$, this becomes
\begin{align}
\{\pm 11,\pm 3\},\qquad \{\pm 9,\pm 7\},
\end{align}
and
\begin{align}
11^2+3^2=9^2+7^2.
\end{align}
\end{remark}

 The following reduction is the basic
structural observation used throughout the paper.

\begin{proposition}[Symmetric reduction]
\label{prop:symmetric-reduction}
Let $(A,B)$ be symmetric with center $c\in\frac12\mathbb Z$.  Then, after
subtracting $c$ and multiplying all centered coordinates by $2$, one may write
\begin{align}
2(A-c)=\{\pm x,\pm y\},\qquad
2(B-c)=\{\pm u,\pm v\},
\end{align}
where $x,y,u,v\in\mathbb Z$ have the same parity.  In these coordinates,
$(A,B)$ is an ideal degree-three Prouhet--Tarry--Escott solution if and only
if
\begin{align}
x^2+y^2=u^2+v^2.
\end{align}
\end{proposition}

\begin{proof}
Translation by $c$ preserves equality of power sums, and multiplying all
centered coordinates by $2$ multiplies the $\ell$-th power sums by
$2^\ell$.  Hence these operations do not affect the
Prouhet--Tarry--Escott equalities.

The centered doubled multisets have the form
\begin{align}
\{\pm x,\pm y\},\qquad \{\pm u,\pm v\}.
\end{align}
Their first and third power sums vanish identically.  Thus only the quadratic
condition remains, and it is
\begin{align}
2x^2+2y^2=2u^2+2v^2,
\end{align}
equivalently
\begin{align}
x^2+y^2=u^2+v^2.
\end{align}

The parity statement follows from integrality of the original entries.  If
$c\in\mathbb Z$, then the centered coordinates are integers, so the doubled
coordinates are even.  If $c\in\mathbb Z+\frac12$, then the centered
coordinates are half-integers with odd doubles.  This proves the claim.
\end{proof}

Thus the full symmetric locus is governed by the same quadratic equation in
both parity classes.  To avoid overcounting, we choose representatives with
\begin{align}
0\leq x\leq y,\qquad 0\leq u\leq v,
\end{align}
and exclude the trivial case
\begin{align}
\{x,y\}=\{u,v\}.
\end{align}

For counting purposes, the centered height of a symmetric solution with center
$c$ is
\begin{align}
\operatorname{ht}_{\mathrm{cen}}(A,B;c)
=
\max\{|a-c|,\ |b-c|: a\in A,\ b\in B\}.
\end{align}
Equivalently, it is the ordinary height after subtracting the center.

For the integer-centered class, the doubled centered coordinates are even, so
it is convenient to divide them by $2$. Thus centered configurations of
centered height at most $H_c$ are counted by ordinary nonnegative
representations
\begin{align}
x^2+y^2=u^2+v^2,\qquad 0\leq x,y,u,v\leq H_{c}.
\end{align}
For the half-integer-centered class, the doubled centered coordinates are odd. Keeping these doubled coordinates, centered configurations of centered height at most $H_{c}$ are counted by odd representations
\begin{align}
x^2+y^2=u^2+v^2,\qquad
1\leq x,y,u,v\leq 2H_{c}
\end{align}
with $x,y,u,v$ odd.
This convention will be used in the centered counts of
Section~\ref{sec:counting} and in the summation over admissible centers in
Section~\ref{sec:law}.

\section{An Explicit Infinite Family}
\label{sec:family}

Before counting the full symmetric locus, we record a simple infinite family
of nontrivial symmetric solutions.  This construction plays two roles.  First,
it gives an elementary source of nontrivial ideal degree-three
Prouhet--Tarry--Escott solutions.  Second, it exhibits the multiplicative
structure that will reappear in the counting arguments below.  This family
lies in the integer-centered parity class of the full symmetric locus.

\begin{proposition}
Let $p>r\geq 2$ be integers, and define
\begin{align}
A=\{\pm(pr-1),\pm(p+r)\}, \qquad
B=\{\pm(pr+1),\pm(p-r)\}.
\end{align}
Then $A$ and $B$ form a nontrivial degree-three
Prouhet--Tarry--Escott solution.
\end{proposition}

\begin{proof}
We use the standard two-square identity
\begin{align}
(\alpha-\beta)^2+(\gamma+\delta)^2
=
(\alpha+\beta)^2+(\gamma-\delta)^2
\end{align}
whenever
\begin{align}
\alpha\beta=\gamma\delta.
\end{align}
Indeed, subtracting the right-hand side from the left-hand side gives
\begin{align}
-4\alpha\beta+4\gamma\delta,
\end{align}
which vanishes under this hypothesis.

Now take
\begin{align}
\alpha=pr,\qquad \beta=1,\qquad \gamma=p,\qquad \delta=r.
\end{align}
Then $\alpha\beta=\gamma\delta=pr$, and hence
\begin{align}
(pr-1)^2+(p+r)^2=(pr+1)^2+(p-r)^2.
\end{align}
By the symmetric reduction of Section~\ref{sec:preliminaries}, the multisets
\begin{align}
A=\{\pm(pr-1),\pm(p+r)\}, \qquad
B=\{\pm(pr+1),\pm(p-r)\}
\end{align}
therefore satisfy the degree-three Prouhet--Tarry--Escott conditions.

It remains only to note that the solution is nontrivial.  Since $p>r\geq 2$,
we have
\begin{align}
pr+1>pr-1
\end{align}
and
\begin{align}
pr+1>p+r,
\end{align}
the latter being equivalent to
\begin{align}
(p-1)(r-1)>0.
\end{align}
Thus the element $pr+1$ occurs in $B$ but not in $A$.  Hence
$A\neq B$ as multisets.
\end{proof}

This construction already gives many solutions of bounded height.  Let $C_{\mathrm{constr}}(H)$ denote the number of solutions arising from the above construction with height at most $H$, counted by parameter pairs
$(p,r)$ satisfying $p>r\geq 2$.

\begin{proposition}
As $H\to\infty$,
\begin{align}
C_{\mathrm{constr}}(H)\gg H\log H.
\end{align}
\end{proposition}

\begin{proof}
For each pair of integers $p>r\geq 2$, the construction gives a nontrivial
symmetric solution whose entries are
\begin{align}
\pm(pr-1),\qquad \pm(p+r),\qquad \pm(pr+1),\qquad \pm(p-r).
\end{align}
If
\begin{align}
pr\leq H-1,
\end{align}
then
\begin{align}
|pr\pm 1|\leq H.
\end{align}
Moreover, since $p,r\geq 1$, we have
\begin{align}
p+r\leq pr+1\leq H,
\end{align}
because $p+r\leq pr+1$ is equivalent to
\begin{align}
(p-1)(r-1)\geq 0.
\end{align}
Finally,
\begin{align}
|p-r|\leq p+r\leq H.
\end{align}
Thus every pair $p>r\geq 2$ with $pr\leq H-1$ produces a constructed
solution of height at most $H$.

It follows that
\begin{align}
C_{\mathrm{constr}}(H)
\geq
\#\{(p,r)\in\mathbb Z^2:p>r\geq 2,\ pr\leq H-1\}.
\end{align}
Put $X=H-1$.  Then
\begin{align}
\#\{(p,r):p>r\geq 2,\ pr\leq X\}
=
\sum_{2\leq r\leq \sqrt X}
\left(\left\lfloor \frac Xr\right\rfloor-r\right).
\end{align}
The right-hand side is
\begin{align}
\geq
\sum_{2\leq r\leq \sqrt X}
\left(\frac Xr -r -1\right)
=
X\sum_{2\leq r\leq \sqrt X}\frac1r
-
\sum_{2\leq r\leq \sqrt X} r
-
\sum_{2\leq r\leq \sqrt X} 1
\gg
X\log X.
\end{align}
Therefore
\begin{align}
C_{\mathrm{constr}}(H)\gg H\log H.
\end{align}
\end{proof}

The proposition shows that symmetric solutions are not isolated examples.
They occur in a family parametrized by restricted integer factorizations, and
their abundance is already visible through the elementary hyperbola count
\begin{align}
p>r\geq 2,\qquad pr\leq H.
\end{align}
In the next sections we move beyond this explicit subfamily and count all
centered symmetric configurations.  The resulting count is larger, reflecting
the full arithmetic of representations of integers as sums of two squares.

\section{Counting the Centered Symmetric Locus}
\label{sec:counting}

We now count the centered symmetric locus.  By
Proposition~\ref{prop:symmetric-reduction}, both center classes are governed
by the equation
\begin{align}
x^2+y^2=u^2+v^2,
\end{align}
with a parity condition.  
Throughout Section~\ref{sec:counting}, we write $H$ for the centered-height cutoff $H_c$.
The integer-centered class corresponds, after
dividing the even doubled coordinates by $2$, to ordinary nonnegative
representations in the box $0\leq x,y,u,v\leq H$.  The
half-integer-centered class corresponds to odd doubled coordinates in the box
$1\leq x,y,u,v\leq 2H$.

We first count the integer-centered class with an explicit leading constant,
then obtain the half-integer-centered class as a parity-restricted companion
count.  The two classes have the same leading asymptotic.

\subsection{The Integer-centered Class}

Throughout this subsection we use the ordering convention
\begin{align}
0\leq x\leq y,\qquad 0\leq u\leq v,
\end{align}
and exclude the trivial case
\begin{align}
\{x,y\}=\{u,v\}.
\end{align}
Thus a nontrivial centered symmetric solution in the integer-centered class is
determined by two distinct unordered nonnegative pairs $\{x,y\}$ and
$\{u,v\}$ satisfying
\begin{align}
x^2+y^2=u^2+v^2.
\end{align}

For $H\geq 1$ and $n\geq 0$, define
\begin{align}
s_H(n)
=
\#\bigl\{\{x,y\}:0\leq x\leq y\leq H,\ x^2+y^2=n\bigr\}.
\end{align}
Thus $s_H(n)$ counts unordered nonnegative representations of $n$ as a
sum of two squares inside the box $0\leq x\leq y\leq H$. In particular,
$s_H(0)=1$.

Let $C_{\mathbb Z}(H)$ denote the number of nontrivial centered symmetric
solutions in the integer-centered class, with centered height at most $H$, counted with unordered pairs of sides; that is, $(A,B)$ and $(B,A)$ are identified.
Then
\begin{align}
C_{\mathbb Z}(H)
=
\frac12
\sum_{n\leq 2H^2}
s_H(n)\bigl(s_H(n)-1\bigr).
\end{align}
The factor $1/2$ accounts for interchanging the two representations assigned
to $A$ and $B$.

\subsection{The Second Moment Input for Sums of Two Squares}
\label{subsec:r2-second-moment}

For $n\geq 1$, let
\begin{align}
r_2(n)=\#\{(a,b)\in\mathbb Z^2:a^2+b^2=n\}
\end{align}
denote the usual sum-of-two-squares representation function, where signs and
order are counted.  Thus $r_2(n)$ counts ordered signed representations,
whereas $s_H(n)$ counts unordered nonnegative representations inside the box
$0\leq x\leq y\leq H$.
For background on the classical theory of
representations as sums of two squares, see, for example, \cite{hardy75}.

We use the standard second-moment estimate
\begin{align}
\sum_{n\leq X}r_2(n)^2\asymp X\log X.
\end{align}
More precisely, Borwein and Choi~\cite{BorweinChoi2003} prove the asymptotic
expansion
\begin{align}
\sum_{n\leq X}r_2(n)^2
=
4X\log X+4\alpha_{\mathrm{BC}} X+O(X^{2/3}),
\end{align}
where
\begin{align}
\alpha_{\mathrm{BC}}
=
2\gamma+\frac{8}{\pi}L'_{-4}(1)
-\frac{12}{\pi^2}\zeta'(2)
+\frac13\log 2-1.
\end{align}
This second moment is the source of the logarithmic enhancement in the
symmetric count.

\begin{lemma}
\label{lem:sH-second-moment}
As $H\to\infty$,
\begin{align}
\sum_{n\leq 2H^2}s_H(n)^2\asymp H^2\log H.
\end{align}
\end{lemma}

\begin{proof}
We first prove the upper bound.  Every unordered nonnegative representation
counted by $s_H(n)$ gives at least one ordered signed representation counted
by $r_2(n)$.
Hence
\begin{align}
s_H(n)\leq r_2(n).
\end{align}
Therefore, by the second-moment estimate for $r_2$,
\begin{align}
\sum_{n\leq 2H^2}s_H(n)^2
\leq
\sum_{n\leq 2H^2}r_2(n)^2
\ll H^2\log H.
\end{align}

For the lower bound, we restrict to $n\leq H^2$.  Then every representation
$n=a^2+b^2$ satisfies $|a|,|b|\leq H$, so all representations lie inside
the height box.  Apart from the boundary cases $a=0$, $b=0$, and
$|a|=|b|$, each unordered nonnegative representation
$\{x,y\}$ with $0<x<y\leq H$ contributes exactly eight ordered signed
representations:
\begin{align}
(\pm x,\pm y),\qquad (\pm y,\pm x).
\end{align}
The boundary cases contribute only $O(1)$ ordered signed representations
for each $n$.  Consequently, for $n\leq H^2$,
\begin{align}
r_2(n)\leq 8s_H(n).
\end{align}
It follows that
\begin{align}
r_2(n)^2 \leq 64 s_H(n)^2.
\end{align}
Summing over $n\leq H^2$, we obtain
\begin{align}
\sum_{n\leq H^2}r_2(n)^2
\leq
64 \sum_{n\leq H^2}s_H(n)^2.
\end{align}
Since
\begin{align}
\sum_{n\leq H^2}r_2(n)^2\asymp H^2\log H,
\end{align}
this gives
\begin{align}
\sum_{n\leq H^2}s_H(n)^2\gg H^2\log H.
\end{align}
Finally,
\begin{align}
\sum_{n\leq 2H^2}s_H(n)^2
\geq
\sum_{n\leq H^2}s_H(n)^2,
\end{align}
so the desired lower bound follows.  Combining the upper and lower bounds
proves the lemma.
\end{proof}

\begin{theorem}
\label{thm:centered-integer-class-order}
The integer-centered class satisfies
\begin{align}
C_{\mathbb Z}(H)\asymp H^2\log H.
\end{align}
\end{theorem}

\begin{proof}
By definition,
\begin{align}
C_{\mathbb Z}(H)
=
\frac12
\sum_{n\leq 2H^2}
s_H(n)\bigl(s_H(n)-1\bigr).
\end{align}
Equivalently,
\begin{align}
C_{\mathbb Z}(H)
=
\frac12
\sum_{n\leq 2H^2}s_H(n)^2
-
\frac12
\sum_{n\leq 2H^2}s_H(n).
\end{align}
By Lemma~\ref{lem:sH-second-moment},
\begin{align}
\sum_{n\leq 2H^2}s_H(n)^2\asymp H^2\log H.
\end{align}
On the other hand,
\begin{align}
\sum_{n\leq 2H^2}s_H(n)
=
\#\{(x,y)\in\mathbb Z^2:0\leq x\leq y\leq H\}
=
\frac{(H+1)(H+2)}2
=
O(H^2).
\end{align}
Since $H^2=o(H^2\log H)$, the linear term is negligible compared with the
second moment.  Hence
\begin{align}
C_{\mathbb Z}(H)\asymp H^2\log H.
\end{align}
\end{proof}

The logarithmic factor reflects the second moment of the
sum-of-two-squares representation function.  Thus the integer-centered
symmetric class is already substantially larger than the explicit family
constructed in Section~\ref{sec:family}.  We next refine this order of
magnitude to an asymptotic formula with an explicit leading constant.

\subsection{A Box Asymptotic with Leading Constant}
\label{subsec:box-asymptotic}

The point is that the height condition used here is box-shaped rather than
radial: we count quadruples satisfying
\begin{align}
|x|,|y|,|u|,|v|\leq H,
\end{align}
rather than representations with
\begin{align}
x^2+y^2\leq H^2.
\end{align}
Accordingly, the leading constant is most naturally obtained by counting
integral points on the split quadric
\begin{align}
x^2+y^2=u^2+v^2
\end{align}
inside the box $[-H,H]^4$.

Let
\begin{align}
T(H)
=
\#\left\{
(x,y,u,v)\in\mathbb Z^4:
\begin{array}{c}
|x|,|y|,|u|,|v|\leq H,\\[1mm]
x^2+y^2=u^2+v^2
\end{array}
\right\}.
\end{align}
Thus $T(H)$ counts signed ordered representations on both sides.

\begin{proposition}
\label{prop:box-quadric-count}
As $H\to\infty$,
\begin{align}
T(H)
=
\frac{128\log 2}{\pi^2}H^2\log H
+
O(H^2).
\end{align}
\end{proposition}

\begin{proof}
Write
\begin{align}
\mathbf s=(x+u,y+v),
\qquad
\mathbf t=(x-u,y-v).
\end{align}
Then
\begin{align}
x^2+y^2-u^2-v^2
=
\mathbf s\cdot \mathbf t.
\end{align}
Hence the quadric equation is equivalent to
\begin{align}
\mathbf s\cdot \mathbf t=0.
\end{align}

We first discard the degenerate cases $\mathbf s=\mathbf 0$ or
$\mathbf t=\mathbf 0$.  These correspond respectively to
\begin{align}
(x,y)=-(u,v)
\qquad\text{and}\qquad
(x,y)=(u,v),
\end{align}
and together contribute only $O(H^2)$ points.  We may therefore assume
that $\mathbf s\neq \mathbf 0$ and $\mathbf t\neq \mathbf 0$.

Let $(m,n)$ be a primitive integer vector, that is, $\gcd(m,n)=1$,
parallel to $\mathbf s$.  Since $\mathbf t$ is orthogonal to
$\mathbf s$, there exist nonzero integers $g,h$ such that
\begin{align}
\mathbf s=(gm,gn),
\qquad
\mathbf t=(-hn,hm).
\end{align}
Consequently,
\begin{align}
2x=gm-hn,
\qquad
2y=gn+hm,
\end{align}
and
\begin{align}
2u=gm+hn,
\qquad
2v=gn-hm.
\end{align}
The parity conditions ensuring that $x,y,u,v$ are integers are
\begin{align}
gm\equiv hn\pmod 2,
\qquad
gn\equiv hm\pmod 2.
\end{align}

Because $(m,n)$ is primitive, $m$ and $n$ cannot both be even.  Hence
there are only two parity cases: either $m,n$ have opposite parity, or
$m,n$ are both odd.  Define
\begin{align}
\delta(m,n)
=
\begin{cases}
\frac14, & m,n \text{ have opposite parity},\\[1mm]
\frac12, & m,n \text{ are both odd}.
\end{cases}
\end{align}
Indeed, if $m,n$ have opposite parity, then $g$ and $h$ must both be
even; if $m,n$ are both odd, the condition is equivalent to
$g\equiv h\pmod 2$.

Thus $\delta(m,n)$ is the density of the admissible parity sublattice
$\Lambda_{m,n}$ in the
$(g,h)$-plane.

The box inequalities are
\begin{align}
|gm-hn|\leq 2H,
\qquad
|gn+hm|\leq 2H,
\end{align}
\begin{align}
|gm+hn|\leq 2H,
\qquad
|gn-hm|\leq 2H.
\end{align}
Equivalently,
\begin{align}
|gm|+|hn|\leq 2H,
\qquad
|gn|+|hm|\leq 2H.
\end{align}
For fixed $m\geq n>0$, let $\mathscr R_{m,n}(H)$ denote the region in the
$(g,h)$-plane cut out by these two inequalities.  It is a centrally
symmetric polygon.  Restricting first to the quadrant $g,h > 0$, its upper
boundary is given by
\begin{align}
h\leq \min\left\{\frac{2H-mg}{n},\frac{2H-ng}{m}\right\}.
\end{align}
The two lines meet at
\begin{align}
g=h=\frac{2H}{m+n}.
\end{align}
A direct calculation gives
\begin{align}
\operatorname{area}\bigl(\mathscr R_{m,n}(H)\bigr)
=
\frac{16H^2}{m(m+n)}.
\end{align}
By the standard lattice-point estimate for a planar polygon,
\begin{align}
\#\bigl(\mathscr R_{m,n}(H)\cap \Lambda_{m,n}\bigr)
=
\delta(m,n)\frac{16H^2}{m(m+n)}
+
O\!\left(\frac{H}{m}+1\right),
\end{align}
where $\Lambda_{m,n}$ is the parity sublattice determined above.
Removing the points with $g=0$ or $h=0$ changes the count by at most the
same error term.

For a primitive pair $m>n>0$, there are four projective primitive directions
with absolute coordinates $\{m,n\}$, represented by
\begin{align}
(m,n),\quad (m,-n),\quad (n,m),\quad (n,-m),
\end{align}
up to overall sign.  The exceptional directions with $n=0$ or $m=n$
contribute only $O(H^2)$ points in total.  Here and below, $(m,n)=1$
means $\gcd(m,n)=1$.  Summing over all primitive directions therefore gives
\begin{align}
T(H)
=
64H^2
\sum_{\substack{m>n\geq1\\ (m,n)=1\\ m+n\leq 2H}}
\frac{\delta(m,n)}{m(m+n)}
+
O(H^2).
\end{align}
It remains to evaluate the arithmetic sum.  We record the required estimate
separately.

\begin{lemma}
\label{lem:weighted-primitive-sum}
Let
\begin{align}
\delta(m,n)
=
\begin{cases}
\frac14, & (m,n)=1 \text{ and } m,n \text{ have opposite parity},\\[1mm]
\frac12, & (m,n)=1 \text{ and } m,n \text{ are both odd},\\[1mm]
0, & \text{otherwise}.
\end{cases}
\end{align}
Then
\begin{align}
\sum_{\substack{m>n\geq1\\ m+n\leq X}}
\frac{\delta(m,n)}{m(m+n)}
=
\frac{2\log 2}{\pi^2}\log X
+
O(1).
\end{align}
\end{lemma}

\begin{proof}
It is convenient first to replace the condition $m+n\le X$ by $m\le X$.
The difference is supported on $X/2<m\le X$, and there
\begin{align}
\sum_{n<m}\frac1{m(m+n)}\ll \frac1m.
\end{align}
Hence
\begin{align}
\sum_{\substack{m>n\ge 1\\ m+n\le X}}
\frac{\delta(m,n)}{m(m+n)}
=
\sum_{\substack{m>n\ge 1\\ m\le X}}
\frac{\delta(m,n)}{m(m+n)}
+O(1).
\end{align}
We shall therefore work with the latter cutoff.

For $a,b\in\{0,1\}$, define
\begin{align}
U_{a,b}(X)
=
\sum_{\substack{1\le n<m\le X\\ m\equiv a\;(\mathrm{mod}\ 2)\\ n\equiv b\;(\mathrm{mod}\ 2)}}
\frac1{m(m+n)}.
\end{align}
For fixed $m$, one has
\begin{align}
\sum_{\substack{1\le n<m\\ n\equiv b\;(\mathrm{mod}\ 2)}}
\frac1{m+n}
=
\frac12\sum_{1\le n<m}\frac1{m+n}
+O\!\left(\frac1m\right),
\end{align}
uniformly in $b$.  Since
\begin{align}
\sum_{1\le n<m}\frac1{m+n}
=
\log 2+O\!\left(\frac1m\right),
\end{align}
it follows that
\begin{align}
U_{a,b}(X)
=
\frac{\log 2}{2}
\sum_{\substack{m\le X\\ m\equiv a\;(\mathrm{mod}\ 2)}}\frac1m
+O(1).
\end{align}
Using
\begin{align}
\sum_{\substack{m\le X\\ m\equiv a\;(\mathrm{mod}\ 2)}}\frac1m
=
\frac12\log X+O(1),
\end{align}
we obtain, for each $a,b\in\{0,1\}$,
\begin{align}
U_{a,b}(X)
=
\frac14\log 2\,\log X+O(1).
\end{align}

Let
\begin{align}
U_{\mathrm{op}}(X)
=
\sum_{\substack{1\le n<m\le X\\ m,n\ \mathrm{of\ opposite\ parity}}}
\frac1{m(m+n)}
\end{align}
and
\begin{align}
U_{\mathrm{oo}}(X)
=
\sum_{\substack{1\le n<m\le X\\ m,n\ \mathrm{both\ odd}}}
\frac1{m(m+n)}.
\end{align}
From the preceding parity decomposition,
\begin{align}
U_{\mathrm{op}}(X)
=
\frac12\log 2\,\log X+O(1),
\qquad
U_{\mathrm{oo}}(X)
=
\frac14\log 2\,\log X+O(1).
\end{align}

We next impose coprimality.  If $m$ and $n$ have opposite parity, then any
common divisor is odd.  Hence M\"obius inversion gives
\begin{align}
\sum_{\substack{1\le n<m\le X\\ (m,n)=1\\ m,n\ \mathrm{of\ opposite\ parity}}}
\frac1{m(m+n)}
=
\sum_{\substack{d\ge1\\ d\ \mathrm{odd}}}
\frac{\mu(d)}{d^2}\,
U_{\mathrm{op}}\!\left(\frac Xd\right).
\end{align}
Since
\begin{align}
\sum_{\substack{d\ge1\\ d\ \mathrm{odd}}}\frac{\mu(d)}{d^2}
=
\prod_{p\ \mathrm{odd}}\left(1-\frac1{p^2}\right)
= \left(1 - \frac{1}{2^2}\right)^{-1} \frac{1}{\zeta(2)}
=
\frac{8}{\pi^2},
\end{align}
and the series
\begin{align}
\sum_{d\ge1}\frac{|\mu(d)|\log d}{d^2}
\end{align}
converges, we obtain
\begin{align}
\sum_{\substack{1\le n<m\le X\\ (m,n)=1\\ m,n\ \mathrm{of\ opposite\ parity}}}
\frac1{m(m+n)}
=
\frac4{\pi^2}\log 2\,\log X+O(1).
\end{align}
Similarly,
\begin{align}
\sum_{\substack{1\le n<m\le X\\ (m,n)=1\\ m,n\ \mathrm{both\ odd}}}
\frac1{m(m+n)}
=
\frac2{\pi^2}\log 2\,\log X+O(1).
\end{align}

Recalling the definition of $\delta(m,n)$, we therefore have
\begin{align}
\sum_{\substack{m>n\ge1\\ m+n\le X}}
\frac{\delta(m,n)}{m(m+n)}
=
\frac14\cdot \frac4{\pi^2}\log 2\,\log X
+
\frac12\cdot \frac2{\pi^2}\log 2\,\log X
+O(1).
\end{align}
Thus
\begin{align}
\sum_{\substack{m>n\ge1\\ m+n\le X}}
\frac{\delta(m,n)}{m(m+n)}
=
\frac{2\log 2}{\pi^2}\log X+O(1),
\end{align}
as claimed.
\end{proof}

Applying Lemma~\ref{lem:weighted-primitive-sum} with $X=2H$, we obtain
\begin{align}
T(H)
=
64H^2
\left(
\frac{2\log 2}{\pi^2}\log H+O(1)
\right)
+
O(H^2),
\end{align}
and hence
\begin{align}
T(H)
=
\frac{128\log 2}{\pi^2}H^2\log H+O(H^2).
\end{align}
\end{proof}

We now convert the signed ordered count into the multiset convention used for
the integer-centered class.  For $n\geq0$, let
\begin{align}
r_{2,H}(n)
=
\#\{(a,b)\in\mathbb Z^2: |a|,|b|\leq H,\ a^2+b^2=n\}.
\end{align}
Then
\begin{align}
T(H)=\sum_{n\leq 2H^2} r_{2,H}(n)^2.
\end{align}
For a generic unordered nonnegative representation
\begin{align}
n=x^2+y^2,
\qquad
0<x<y\leq H,
\end{align}
there are exactly eight signed ordered representations:
\begin{align}
(\pm x,\pm y),
\qquad
(\pm y,\pm x).
\end{align}
For \(n\geq 1\), the only exceptions arise from the boundary cases
\(x=0\) and \(x=y\).  These correspond respectively to
\(n=t^2\) and \(n=2t^2\), with \(t\in\mathbb Z_{\geq 1}\).
Thus, for \(n\geq 1\),
\begin{align}
r_{2,H}(n)=8s_H(n)-4\epsilon(n),
\end{align}
where
\begin{align}
\epsilon(n)
=
\begin{cases}
1, & \text{if } n=t^2 \text{ or } n=2t^2
\text{ for some } t\in\mathbb Z_{\geq 1},\\[1mm]
0, & \text{otherwise}.
\end{cases}
\end{align}
The term \(n=0\) contributes only \(O(1)\) to \(T(H)\), and may be absorbed
into the error term.  Since \(\epsilon(n)\) is supported only on integers of
the form \(t^2\) or \(2t^2\), there are \(O(H)\) such integers with
\(n\leq 2H^2\).  Also \(s_H(n)\leq H+1\).  Therefore
\begin{align}
T(H)
&=
\sum_{n\leq 2H^2} r_{2,H}(n)^2 \\
&=
64\sum_{1\leq n\leq 2H^2}s_H(n)^2
+
O\!\left(
\sum_{\substack{1\leq n\leq 2H^2\\ \epsilon(n)=1}}
s_H(n)
\right)
+
O(H)
+
O(1) \\
&=
64\sum_{1\leq n\leq 2H^2}s_H(n)^2
+
O(H^2).
\end{align}
Since \(s_H(0)=1\), replacing the range \(1\leq n\leq 2H^2\) by
\(0\leq n\leq 2H^2\), or simply by \(n\leq 2H^2\), changes the right-hand
side by only \(O(1)\).  Hence
\begin{align}
T(H)
=
64\sum_{n\leq 2H^2}s_H(n)^2
+
O(H^2).
\end{align}
Combining this with Proposition~\ref{prop:box-quadric-count} gives
\begin{align}
\sum_{n\leq 2H^2}s_H(n)^2
=
\frac{2\log 2}{\pi^2}H^2\log H
+
O(H^2).
\end{align}

\begin{theorem}
\label{thm:centered-integer-class}
The integer-centered class satisfies
\begin{align}
C_{\mathbb Z}(H)
=
\frac{\log 2}{\pi^2}H^2\log H
+
O(H^2).
\end{align}
\end{theorem}

\begin{proof}
By definition,
\begin{align}
C_{\mathbb Z}(H)
=
\frac12
\sum_{n\leq 2H^2}
s_H(n)\bigl(s_H(n)-1\bigr).
\end{align}
Hence
\begin{align}
C_{\mathbb Z}(H)
=
\frac12
\sum_{n\leq 2H^2}s_H(n)^2
-
\frac12
\sum_{n\leq 2H^2}s_H(n).
\end{align}
The first term satisfies
\begin{align}
\frac12
\sum_{n\leq 2H^2}s_H(n)^2
=
\frac{\log 2}{\pi^2}H^2\log H
+
O(H^2),
\end{align}
by the preceding computation.  The second term is simply
\begin{align}
\sum_{n\leq 2H^2}s_H(n)
=
\#\{(x,y)\in\mathbb Z^2:0\leq x\leq y\leq H\}
=
\frac{(H+1)(H+2)}2
=
O(H^2).
\end{align}
Therefore
\begin{align}
C_{\mathbb Z}(H)
=
\frac{\log 2}{\pi^2}H^2\log H
+
O(H^2).
\end{align}
\end{proof}

The constant $\log 2/\pi^2$ reflects the geometry of the height box.  In
contrast with a radial count ordered by $x^2+y^2$, the present problem
samples all representations with $0\leq x,y\leq H$, including the outer
annular range $H^2<n\leq 2H^2$.  The factor $\log 2$ is the analytic trace
of this box geometry, arising from the integral
\begin{align}
\int_0^1 \frac{dt}{1+t}=\log 2.
\end{align}

For comparison, Appendix~\ref{app:finite-sqSum-enumeration} carries out the corresponding centered enumeration under a radial cutoff $Q(A;c)=Q(B;c)\le U$, where the leading constant changes because the underlying cutoff region is a disk rather than a square.

\subsection{A Parity-Restricted Companion Count}

Let $C_{\mathbb Z+\frac12}(H)$ denote the number of nontrivial centered
symmetric solutions in the half-integer-centered class, with centered height
at most $H$.  After subtracting the center and multiplying by $2$, such
solutions are counted by odd variables
\begin{align}
1\leq x\leq y\leq 2H,\qquad
1\leq u\leq v\leq 2H,
\end{align}
satisfying
\begin{align}
x^2+y^2=u^2+v^2.
\end{align}

\begin{proposition}
\label{prop:half-integer-class}
As $H\to\infty$,
\begin{align}
C_{\mathbb Z+\frac12}(H)
=
\frac{\log 2}{\pi^2}H^2\log H+O(H^2).
\end{align}
\end{proposition}

\begin{proof}
Let
\begin{align}
T_{\mathrm{odd}}(M)
=
\#\left\{
(x,y,u,v)\in(2\mathbb Z+1)^4:
\begin{array}{c}
|x|,|y|,|u|,|v|\leq M,\\
x^2+y^2=u^2+v^2
\end{array}
\right\}.
\end{align}
The proof of Proposition~\ref{prop:box-quadric-count} applies with the
additional condition that $x,y,u,v$ are odd.  In the parametrization used
there, this is a congruence condition modulo \(4\) on the same admissible
\((g,h)\)-lattice.  For each primitive direction \((m,n)\), it selects
congruence classes of relative density \(1/4\) inside the previously
admissible parity classes, and the boundary contribution remains \(O(M^2)\)
after summing over primitive directions.  Consequently,
\begin{align}
T_{\mathrm{odd}}(M)
=
\frac{32\log 2}{\pi^2}M^2\log M+O(M^2).
\end{align}
Taking $M=2H$, we obtain
\begin{align}
T_{\mathrm{odd}}(2H)
=
\frac{128\log 2}{\pi^2}H^2\log H+O(H^2).
\end{align}
Converting from signed ordered representations to unordered nonnegative
representations divides the leading term by $64$, exactly as in the
integer-centered class.  Hence
\begin{align}
\sum_n \bigl(s_H^{\mathrm{odd}}(n)\bigr)^2
=
\frac{2\log 2}{\pi^2}H^2\log H+O(H^2),
\end{align}
where $s_H^{\mathrm{odd}}(n)$ counts unordered odd representations
\begin{align}
n=x^2+y^2,\qquad 1\leq x\leq y\leq 2H,\qquad x,y\ \text{odd}.
\end{align}
The linear term
\begin{align}
\sum_n s_H^{\mathrm{odd}}(n)
\end{align}
counts unordered odd pairs $1\leq x\leq y\leq 2H$, and is therefore
$O(H^2)$. Hence subtracting it only changes the estimate by $O(H^2)$.
Therefore
\begin{align}
C_{\mathbb Z+\frac12}(H)
=
\frac12\sum_n s_H^{\mathrm{odd}}(n)
\bigl(s_H^{\mathrm{odd}}(n)-1\bigr)
=
\frac{\log 2}{\pi^2}H^2\log H+O(H^2).
\end{align}
\end{proof}

\subsection{The Full Centered Symmetric Count}

\begin{theorem}
\label{thm:centered-symmetric-asymptotic}
Let $C_{\mathrm{sym}}(H)$ denote the number of nontrivial centered
symmetric configurations of centered height at most $H$, counted with
unordered multiset conventions and including both center classes.  Then
\begin{align}
C_{\mathrm{sym}}(H)
=
\frac{2\log 2}{\pi^2}H^2\log H+O(H^2).
\end{align}
\end{theorem}

\begin{proof}
The two center classes are disjoint, and
\begin{align}
C_{\mathrm{sym}}(H)
=
C_{\mathbb Z}(H)+C_{\mathbb Z+\frac12}(H).
\end{align}
The result follows from Theorem~\ref{thm:centered-integer-class} and
Proposition~\ref{prop:half-integer-class}.
\end{proof}

\section{Translations and the Symmetric Counting Law}
\label{sec:law}

We now pass from centered symmetric configurations to integer solutions with
arbitrary admissible center.  This accounts for the translation freedom in the
Prouhet--Tarry--Escott problem and gives the counting law for symmetric
solutions of bounded height.

We treat the two center classes separately.  In the integer-centered class, a
centered configuration has the form
\begin{align}
A_0=\{\pm x,\pm y\},\qquad B_0=\{\pm u,\pm v\},
\end{align}
with
\begin{align}
0\leq x\leq y,\qquad 0\leq u\leq v,
\end{align}
and centered height
\begin{align}
R=\max\{x,y,u,v\}.
\end{align}
For an integer $d$, define
\begin{align}
A_d=\{d\pm x,d\pm y\},\qquad B_d=\{d\pm u,d\pm v\}.
\end{align}
Since the Prouhet--Tarry--Escott property is invariant under translations,
$(A_d,B_d)$ is again an ideal degree-three solution.

The height of $(A_d,B_d)$ is at most $H$ if and only if all its entries lie
in $[-H,H]$.  Since the largest absolute centered entry is $R$, this is
equivalent to
\begin{align}
|d|+R\leq H.
\end{align}
Thus, for a fixed centered configuration of centered height $R\leq H$, the admissible
integer centers are precisely
\begin{align}
-H+R\leq d\leq H-R,
\end{align}
and their number is
\begin{align}
2(H-R)+1.
\end{align}

Let $\Delta C_{\mathbb Z}(R)$ denote the number of integer-centered
configurations of exact centered height $R$, so that
\begin{align}
\Delta C_{\mathbb Z}(R)
=
C_{\mathbb Z}(R)-C_{\mathbb Z}(R-1).
\end{align}
Then
\begin{align}
N_{\mathbb Z}(H)
=
\sum_{R\leq H}
\bigl(2(H-R)+1\bigr)\Delta C_{\mathbb Z}(R).
\end{align}
A discrete summation by parts gives
\begin{align}
N_{\mathbb Z}(H)
=
2\sum_{R=0}^{H-1}C_{\mathbb Z}(R)
+ C_{\mathbb Z}(H).
\end{align}
Using Theorem~\ref{thm:centered-integer-class}, we obtain
\begin{align}
N_{\mathbb Z}(H)
=
\frac{2\log 2}{3\pi^2}H^3\log H+O(H^3).
\end{align}

We next consider the half-integer-centered class.  A centered configuration
has the form
\begin{align}
A_0=\left\{\pm\frac{x}{2},\pm\frac{y}{2}\right\},
\qquad
B_0=\left\{\pm\frac{u}{2},\pm\frac{v}{2}\right\},
\end{align}
where $x,y,u,v$ are odd integers and
\begin{align}
1\leq x\leq y,\qquad 1\leq u\leq v.
\end{align}
Its centered height is
\begin{align}
R=\frac12\max\{x,y,u,v\}.
\end{align}
For an odd integer $d$, define
\begin{align}
A_d=
\left\{\frac{d\pm x}{2},\frac{d\pm y}{2}\right\},
\qquad
B_d=
\left\{\frac{d\pm u}{2},\frac{d\pm v}{2}\right\}.
\end{align}
The condition that all entries lie in $[-H,H]$ is
\begin{align}
|d|+2R\leq 2H.
\end{align}
Since $x,y,u,v$ are odd, the quantity $2R$ is odd.  Hence
$2(H-R)$ is an odd integer, and the admissible odd integers $d$ are
precisely those satisfying
\begin{align}
|d|\leq 2(H-R).
\end{align}
Their number is therefore
\begin{align}
2(H-R)+1.
\end{align}

Let $\Delta C_{\mathbb Z+\frac12}(R)$ denote the number of
half-integer-centered configurations of exact centered height $R$, where
$R$ ranges over the half-integers.  Then
\begin{align}
N_{\mathbb Z+\frac12}(H)
=
\sum_{\substack{R\leq H\\ R\in \mathbb Z+\frac12}}
\bigl(2(H-R)+1\bigr)
\Delta C_{\mathbb Z+\frac12}(R).
\end{align}
Equivalently, by summation by parts over the half-integer height values,
\begin{align}
N_{\mathbb Z+\frac12}(H)
=
2
\sum_{\substack{0<R\leq H\\ R\in \mathbb Z+\frac12}}
C_{\mathbb Z+\frac12}(R).
\end{align}
Using Proposition~\ref{prop:half-integer-class}, this gives
\begin{align}
N_{\mathbb Z+\frac12}(H)
=
\frac{2\log 2}{3\pi^2}H^3\log H+O(H^3).
\end{align}

Combining the two center classes gives the shifted symmetric counting law.

\begin{theorem}
\label{thm:Nsym}
Let $N_{\mathrm{sym}}(H)$ denote the number of nontrivial symmetric ideal
degree-three Prouhet--Tarry--Escott solutions of size four and height at most
$H$, with distinct admissible centers counted separately.  Then
\begin{align}
N_{\mathrm{sym}}(H)
=
\frac{4\log 2}{3\pi^2}H^3\log H+O(H^3).
\end{align}
\end{theorem}

\begin{proof}
By the preceding estimates,
\begin{align}
N_{\mathbb Z}(H)
=
\frac{2\log 2}{3\pi^2}H^3\log H+O(H^3)
\end{align}
and
\begin{align}
N_{\mathbb Z+\frac12}(H)
=
\frac{2\log 2}{3\pi^2}H^3\log H+O(H^3).
\end{align}
Since the two center classes are disjoint,
\begin{align}
N_{\mathrm{sym}}(H)
=
N_{\mathbb Z}(H)+N_{\mathbb Z+\frac12}(H).
\end{align}
Therefore
\begin{align}
N_{\mathrm{sym}}(H)
=
\frac{4\log 2}{3\pi^2}H^3\log H+O(H^3).
\end{align}
\end{proof}

This theorem shows that summing over admissible centers enlarges the centered
symmetric count by one power of $H$, while preserving the logarithmic
enhancement coming from the sum-of-two-squares structure.  The essential
arithmetic is already present in the centered equation
\begin{align}
x^2+y^2=u^2+v^2,
\end{align}
and the center parameter converts the centered asymptotic into the shifted
asymptotic above.

\section{Enumeration Evidence and a Polynomial-Scale Density Conjecture}
\label{sec:conjecture}

The estimates proved above show that the symmetric locus is large: after summing over admissible centers, it contains
\begin{align}
N_{\mathrm{sym}}(H)
=
\frac{4\log 2}{3\pi^2}H^3\log H+O(H^3)
\end{align}
solutions of height at most $H$. A natural next question is how this family compares with the full set of ideal degree-three Prouhet--Tarry--Escott solutions of bounded height.

Let $N(H)$ denote the number of nontrivial ideal degree-three Prouhet--Tarry--Escott solutions of size four and height at most $H$, counted with the same conventions used for $N_{\mathrm{sym}}(H)$. Thus permutations inside each multiset are not counted separately, the two multisets are counted up to interchange, and distinct integer translates are counted separately. Let
\begin{align}
R(H)=\frac{N_{\mathrm{sym}}(H)}{N(H)}
\end{align}
denote the corresponding proportion of symmetric solutions.

The preceding sections give a structural reason to expect the symmetric locus to remain visible inside the full height-ordered solution space. General degree-three solutions are constrained by the three coupled equations
\begin{align}
\sum_i a_i=\sum_i b_i,\qquad
\sum_i a_i^2=\sum_i b_i^2,\qquad
\sum_i a_i^3=\sum_i b_i^3.
\end{align}
By contrast, the symmetric locus reduces after centering to the single quadratic equation
\begin{align}
x^2+y^2=u^2+v^2.
\end{align}
This reduction places symmetric solutions in the arithmetic setting of representations as sums of two squares, and the second-moment estimate for $r_2(n)$ explains the logarithmic enhancement in their count.

Finite-height enumeration is consistent with this picture. In computations over bounded boxes, symmetric solutions form a substantial fraction of all solutions at small and moderate height. Although the observed proportion decreases over the computed range, the decline is gradual; moreover, the symmetric and non-symmetric counts appear to have comparable polynomial growth. This suggests that the symmetric locus may have the same height exponent as the full ideal degree-three solution space, even if its literal density might tend to zero.

A companion centered enumeration with a finite-squared-sum cutoff is given in Appendix~\ref{app:finite-sqSum-enumeration}; it is included to compare the radial cutoff geometry with the height-box geometry used in the main theorems.

These observations motivate the following distributional conjecture.

\begin{conjecture}[Polynomial-scale density of the symmetric locus]
\label{conj:polynomial-density}
Let $N(H)$ denote the number of nontrivial ideal degree-three Prouhet--Tarry--Escott solutions of size four and height at most $H$, counted with the conventions above. Then
\begin{align}
N(H)=H^{3+o(1)}.
\end{align}
Consequently, since
\begin{align}
N_{\mathrm{sym}}(H)
=
\frac{4\log 2}{3\pi^2}H^3\log H+O(H^3),
\end{align}
the symmetric locus accounts for the full solution space up to subpolynomial factors:
\begin{align}
\frac{N_{\mathrm{sym}}(H)}{N(H)}\geq H^{-o(1)}.
\end{align}
\end{conjecture}

This conjecture does not assert positive limiting density. Instead, it asserts that symmetric solutions have the same polynomial height exponent as all solutions. Thus symmetric solutions may fail to occupy a fixed positive proportion of the solution space, but would still be arithmetically visible at the level of power-law growth.

One may also formulate a sharper logarithmic version. Since
\begin{align}
N_{\mathrm{sym}}(H)
=
\frac{4\log 2}{3\pi^2}H^3\log H+O(H^3),
\end{align}
an estimate of the form
\begin{align}
N(H)\ll H^3(\log H)^A
\end{align}
with $A\geq 1$ would imply
\begin{align}
R(H)\gg (\log H)^{1-A}.
\end{align}
The case $A=1$ would imply positive lower density of the symmetric locus, whereas larger values of $A$ would allow the density to decay at most logarithmically. The polynomial-scale form in Conjecture~\ref{conj:polynomial-density} is weaker and is therefore the more robust conjecture.

\begin{remark}[Arithmetic amplification]
The conjecture above is not suggested by the naive degree-weighted heuristic.
Over $\mathbb R$, the symmetric ansatz
\begin{align}
A=\{\pm x,\pm y\},\qquad B=\{\pm u,\pm v\}
\end{align}
imposes a special pairing structure and therefore appears to define a restricted subfamily of the full degree-three Prouhet--Tarry--Escott solution variety. 
The symmetric condition is
\begin{align}
x^2+y^2=u^2+v^2,
\end{align}
whose solutions are organized by the representation function $r_2(n)$. The second-moment estimate
\begin{align}
\sum_{n\leq X} r_2(n)^2\asymp X\log X
\end{align}
gives the logarithmic enhancement
\begin{align}
C_{\mathrm{sym}}(H)
=
\frac{2 \log 2}{\pi^2}H^2\log H+O(H^2),
\qquad
N_{\mathrm{sym}}(H)
=
\frac{4\log 2}{3\pi^2}H^3\log H+O(H^3).
\end{align}
Thus the symmetric locus is amplified by the arithmetic multiplicity of
representations as sums of two squares. This explains why it can be visible
at the polynomial scale in the integral Prouhet--Tarry--Escott problem.
\end{remark}

A proof of Conjecture~\ref{conj:polynomial-density} would require information about the total number $N(H)$ of ideal degree-three solutions. At present, no asymptotic formula for $N(H)$ is known in the counting convention used here. One possible route would be to obtain matching upper and lower bounds for $N(H)$, or at least an upper bound of order $H^3\log H$ up to subpolynomial factors. Such a result would place the symmetric asymptotic proved in Section~\ref{sec:law} in direct comparison with the full solution space.

There are several natural refinements. One may count primitive solutions, where the entries have no common factor, or quotient further by affine equivalence. One may also ask whether analogous symmetric loci occur in higher-degree ideal Prouhet--Tarry--Escott problems, and whether their counts are again governed by representation functions of comparable arithmetic origin. These questions suggest that symmetry-driven subvarieties may play a broader role in the distribution of solutions to nonlinear Diophantine systems.

From the physical point of view, the conjecture points to a possible organizing principle for anomaly-free charge spectra. Under the degree-three Prouhet--Tarry--Escott correspondence, symmetric solutions give Lee--Takahashi--Tsai (LTT) particle spectra organized into doublet-like pairs, with states sharing nearby masses due to the charge assignments~\cite{Lee:2026djo}. Polynomial-scale abundance of the symmetric locus would therefore indicate that such paired structures are not exceptional coincidences, but a persistent arithmetic feature of anomaly-free spectra shaped by the constraints underlying anomaly cancellation.

\section*{Acknowledgments}

Y.-D.T. is supported by a Dorothy Hodgkin Fellowship funded by the Royal Society, UK, and is grateful for start-up support from the University of Manchester and the University of Sheffield.
This work was supported by JSPS KAKENHI Grant Numbers 25H02165 (F.T.), 25KJ0564 (J.L.), and 26K00695 (F.T.). This work was also supported by the World Premier International Research Center Initiative (WPI), MEXT, Japan, and by COST Action COSMIC WISPers CA21106, supported by COST (European Cooperation in Science and Technology). J.L. was also supported by the Graduate Program on Physics for the Universe (GP-PU), Tohoku University.


\begingroup
\setlength{\bibsep}{0.45em}

\bibliographystyle{utcaps_YTmod}
\bibliography{ref}

\endgroup


\appendix
\markboth{APPENDICES}{APPENDICES}

\section{The Full PTE Space and Vinogradov Mean Values}
\label{sec:vmvt-full-pte}

The symmetric count proved above gives
\begin{align}
N_{\mathrm{sym}}(H)
=
\frac{4\log 2}{3\pi^2}H^3\log H+O(H^3).
\end{align}
Conjecture~\ref{conj:polynomial-density} asks whether the full ideal degree-three Prouhet--Tarry--Escott solution space has the same polynomial height exponent. Let $N(H)$ denote the number of nontrivial ideal degree-three Prouhet--Tarry--Escott solutions of height at most $H$, counted with the multiset conventions of Section~\ref{sec:conjecture}. The conjectural comparison is
\begin{equation}
N(H)=H^{3+o(1)}.
\label{eq:polynomial-density-target}
\end{equation}
Vinogradov's mean-value theorem provides a natural ambient counting framework for this question, although its direct application is too coarse to prove \eqref{eq:polynomial-density-target}.

\subsection{The Vinogradov Ambient Count}
\label{subsec:vinogradov-ambient-count}

For integers $s,k\geq 1$, let $J_{s,k}(X)$ denote the number of integer solutions to the Vinogradov system
\begin{equation}
x_1^j+\cdots+x_s^j
=
y_1^j+\cdots+y_s^j,
\qquad 1\leq j\leq k,
\label{eq:vinogradov-system}
\end{equation}
with
\begin{align}
1\leq x_i,y_i\leq X.
\end{align}
Equivalently,
\begin{equation}
J_{s,k}(X)
=
\int_{[0,1)^k}
\left|
\sum_{1\leq x\leq X}
\exp\!\left[ 2\pi i \left( \alpha_1x+\alpha_2x^2+\cdots+\alpha_kx^k \right) \right]
\right|^{2s}
\,d\boldsymbol{\alpha}.
\label{eq:vmvt-moment}
\end{equation}
Thus $J_{s,k}(X)$ counts pairs of ordered $s$-tuples whose power sums agree through degree $k$.

The ideal degree-three Prouhet--Tarry--Escott problem corresponds to
\begin{align}
s=4,\qquad k=3.
\end{align}
Indeed, before quotienting by permutations, before removing diagonal solutions, and before imposing the multiset conventions of this paper, ordered ideal degree-three solutions are precisely the solutions of
\begin{equation}
\sum_{i=1}^4 a_i^j
=
\sum_{i=1}^4 b_i^j,
\qquad j=1,2,3.
\label{eq:ordered-pte-vmvt}
\end{equation}

The main conjecture in Vinogradov's mean-value theorem, now a theorem
\cite{Wooley2014TheCC,Bourgain2015ProofOT}, gives for every
$\varepsilon>0$
\begin{equation}
J_{s,k}(X)
\ll_{\varepsilon,s,k}
X^{s+\varepsilon}
+
X^{2s-\frac{k(k+1)}2+\varepsilon}.
\label{eq:vmvt-main-bound}
\end{equation}
For $s=4$ and $k=3$, this yields
\begin{equation}
J_{4,3}(X)\ll_\varepsilon X^{4+\varepsilon}.
\label{eq:J43-bound}
\end{equation}

To compare this with the height convention used here, suppose that
\begin{align}
|a_i|,|b_i|\leq H
\end{align}
and that \eqref{eq:ordered-pte-vmvt} holds. After the common translation
\begin{align}
a_i'=a_i+H+1,\qquad b_i'=b_i+H+1,
\end{align}
we have
\begin{align}
1\leq a_i',b_i'\leq 2H+1,
\end{align}
and the equal-power-sum relations remain valid. Hence every ordered integral degree-three solution of height at most $H$ is counted by $J_{4,3}(2H+1)$. Passing from ordered tuples to multisets changes the count by only bounded factors, and removing trivial solutions can only decrease it. Therefore
\begin{equation}
N(H)\ll_\varepsilon H^{4+\varepsilon}.
\label{eq:SH-raw-vmvt}
\end{equation}

This bound is rigorous but does not approach the conjectural scale \eqref{eq:polynomial-density-target}. In particular, combining \eqref{eq:SH-raw-vmvt} with the symmetric asymptotic gives only
\begin{align}
\frac{N_{\mathrm{sym}}(H)}{N(H)}
\gg_\varepsilon
H^{-1-\varepsilon}\log H.
\end{align}
The issue is not a defect of Vinogradov's mean-value theorem. Rather, $J_{4,3}(X)$ is an ambient ordered count: it includes diagonal configurations, counts permutations separately, and does not isolate the affine directions that are intrinsic to the Prouhet--Tarry--Escott problem.

\subsection{Affine Reduction and the Missing Estimate}
\label{subsec:affine-reduced-viewpoint}

The Prouhet--Tarry--Escott equations are invariant under affine transformations
\begin{align}
t\mapsto \xi t+ \eta.
\end{align}
For the height count considered in this paper, the translation parameter is especially visible. For the symmetric locus, Theorem~\ref{thm:centered-symmetric-asymptotic} gives
\begin{align}
C_{\mathrm{sym}}(H)
=
\frac{2 \log 2}{\pi^2}H^2\log H+O(H^2),
\end{align}
and Section~\ref{sec:law} shows that summing over integer translates produces
\begin{align}
N_{\mathrm{sym}}(H)
=
\frac{4\log 2}{3\pi^2}H^3\log H+O(H^3).
\end{align}
Thus, in this locus, one power of $H$ comes from translation, while the essential arithmetic is already present in the centered problem.

For the full integral solution space, one must be slightly more careful. An integral solution need not translate by an integer to an integral centered solution, since its common mean may be nonintegral. Thus integrally centered solutions represent only one affine residue class of the full integral problem. A complete affine reduction would decompose the full solution space according to the common mean modulo $\mathbb Z$, then estimate each resulting class after removing its translation parameter.

This suggests that the natural missing input is not merely a sharper bound for
the raw Vinogradov count, but an affine-reduced estimate for the genuinely
nontrivial centered problem in each relevant residue class.  To make this
precise, let
\begin{equation}
\rho\in\{0,1,2,3\}
\label{eq:affine-residue-index}
\end{equation}
index the affine residue class of the common mean modulo \(\mathbb Z\), so
that the common mean lies in
\begin{equation}
\frac{\rho}{4}+\mathbb Z.
\label{eq:affine-residue-class}
\end{equation}
For this residue class, let \(C_\rho(H)\) denote the corresponding centered
count after subtracting the common mean and removing the integer translation
parameter, with the same nontriviality and multiset conventions as in the
main text.  Let \(N_\rho(H)\) denote the shifted count obtained from this
centered class by summing over all admissible integer translations subject to
the height bound.

The expected affine-reduced estimate is then, uniformly in
\(\rho\in\{0,1,2,3\}\),
\begin{equation}
C_\rho(H)\ll H^2(\log H)^{O(1)}.
\label{eq:desired-centered-bound}
\end{equation}
Such estimates would imply
\begin{equation}
N_\rho(H)\ll H^3(\log H)^{O(1)}
\label{eq:desired-shifted-bound}
\end{equation}
for each affine residue class, and hence
\begin{equation}
N(H)\ll H^3(\log H)^{O(1)}
\label{eq:desired-full-shifted-bound}
\end{equation}
after summing over the four residue classes.  This would place the full
solution space on the same polynomial scale as the symmetric subfamily.  A
bound as sharp as
\begin{equation}
N(H)\ll H^3\log H
\label{eq:sharp-full-shifted-bound}
\end{equation}
would imply positive lower density of the symmetric locus.

\subsection{Outlook}
\label{subsec:analytic-program}

The preceding discussion points to a refinement of the Vinogradov framework adapted to the ideal degree-three Prouhet--Tarry--Escott problem. Such a theory would need to separate diagonal and permutation contributions, remove the translation direction, account for the possible residue classes of the common mean, and then estimate the remaining off-diagonal centered solutions.

The symmetric locus is special because, after centering, the full system collapses to the single quadratic equation
\begin{align}
x^2+y^2=u^2+v^2.
\end{align}
Away from this locus, the problem retains its genuinely mixed-degree character: quadratic and cubic constraints must be handled simultaneously. This suggests that future work on nonsymmetric families may require methods sensitive to mixed cubic--quadratic systems, rather than only equal-degree Vinogradov systems. A useful comparison is provided by circle-method treatments of cubic--quadric intersections, where lower-degree forms cannot simply be discarded during the analytic reduction; see, for example,~\cite{Browning_Dietmann_Heath-Brown_2015}.

\section{Finite-Height Enumeration}
\label{app:finite-enumeration}

In this appendix we describe the finite-height enumeration used to support
the discussion surrounding Conjecture~\ref{conj:polynomial-density}.  The
computations are not used in the proofs of the main theorems.  Their role is
to illustrate the behavior of the height-ordered solution space and to compare
the symmetric count $N_{\mathrm{sym}}(H)$ with the total count $N(H)$ at
accessible heights.

For each height $H$, we enumerate nontrivial ideal degree-three
Prouhet--Tarry--Escott solutions of size four satisfying
\begin{align}
\sum_{i=1}^4 a_i^j=\sum_{i=1}^4 b_i^j,
\qquad j=1,2,3,
\end{align}
with
\begin{align}
|a_i|, |b_i| \leq H.
\end{align}
The counting conventions agree with those used in the main text:
permutations within each multiset are not counted separately, the two
multisets are identified up to interchange, trivial solutions are excluded,
and distinct admissible centers are counted separately.

The symmetric count $N_{\mathrm{sym}}(H)$ used in the figures includes both
center classes
\begin{align}
c\in\mathbb Z
\qquad\text{and}\qquad
c\in\mathbb Z+\frac12.
\end{align}
Thus it is the shifted symmetric count from Theorem~\ref{thm:Nsym}, not only
the integer-centered subcount.

Figure~\ref{fig:SolutionsCounts} compares the symmetric and non-symmetric
counts as functions of $H$.  The gray dashed curve is proportional to the
proved leading term
\begin{align}
\frac{4\log 2}{3\pi^2}H^3\log H
\end{align}
from Theorem~\ref{thm:Nsym}; it is included as a guide to the growth of the
full symmetric locus.  

\begin{figure}
    \centering
    \includegraphics[width=0.70\linewidth]{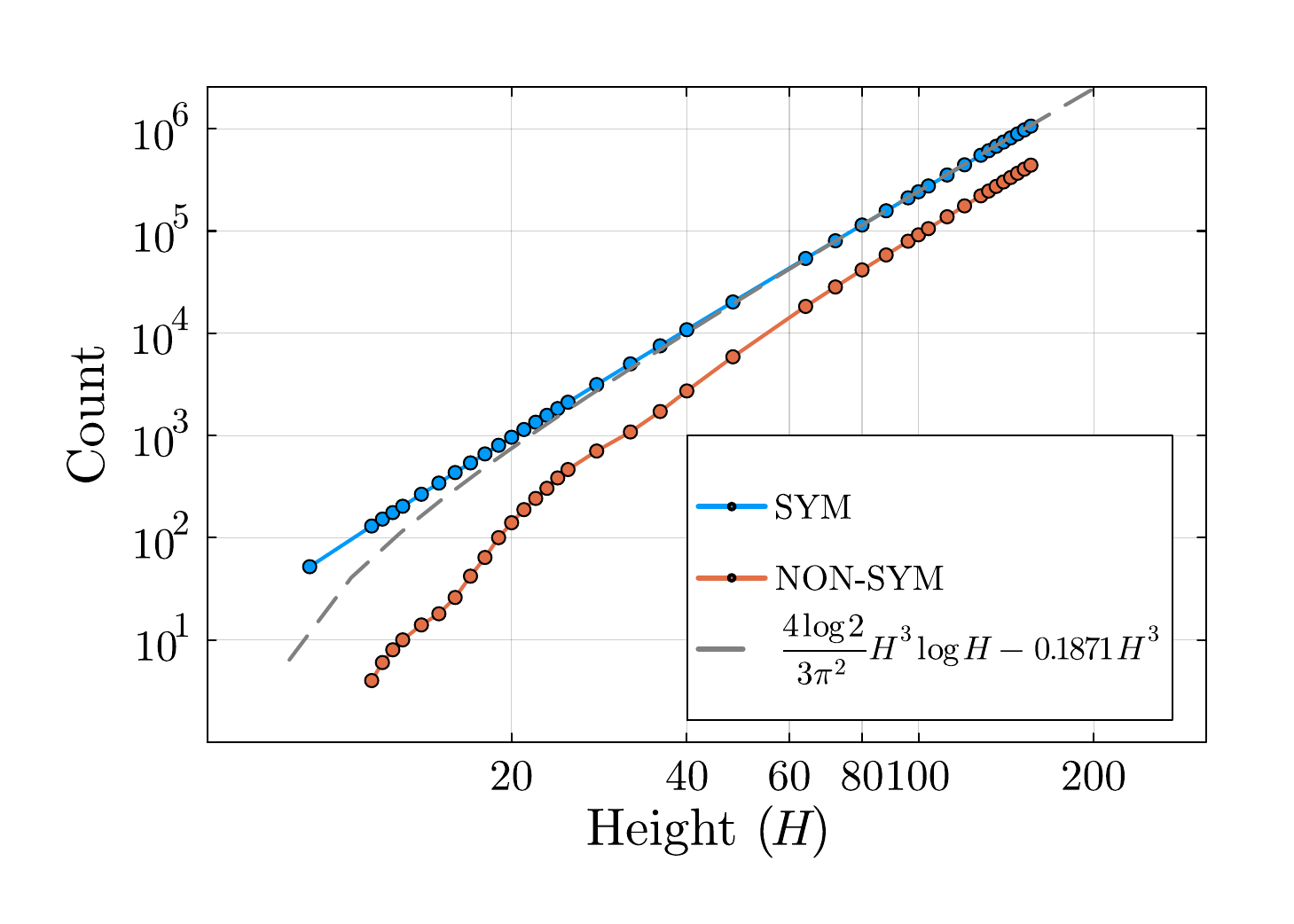}
    \caption{Finite-height enumeration of nontrivial ideal degree-three
Prouhet--Tarry--Escott solutions of size four.
The blue markers give the full symmetric count $N_{\mathrm{sym}}(H)$,
including both center classes $c\in\mathbb Z$ and
$c\in\mathbb Z+\frac12$, while the orange markers give the non-symmetric
count $N_{\mathrm{non\text{-}sym}}(H)$.
All counts use the multiset, interchange, and admissible-center conventions
described in the text.
The gray dashed curve is included as a guide to the proved growth of the full
symmetric locus.
Its leading term is
$\frac{4\log 2}{3\pi^2}H^3\log H$, as obtained in
Theorem~\ref{thm:Nsym}.
We leave the coefficient of the subleading term as a free fitting parameter
to better capture the behavior at finite $H$.
The coefficient used here is obtained by linear regression on the log-log
plane using the data with $H\geq 80$.
The displayed data are compatible with the asymptotic behavior established
for $N_{\mathrm{sym}}(H)$ and show that the symmetric and non-symmetric
counts remain comparable in polynomial scale over the computed range.
    }
    \label{fig:SolutionsCounts}
\end{figure}

Figure~\ref{fig:ratio} shows the observed symmetric proportion
\begin{align}
R(H)=\frac{N_{\mathrm{sym}}(H)}{N(H)}.
\end{align}
Over the computed range, this proportion decreases but remains substantial.
This is consistent with the polynomial-scale density viewpoint of
Conjecture~\ref{conj:polynomial-density}: the symmetric locus may have the
same polynomial height exponent as the full height-ordered solution space,
even if its limiting density is not positive.

\begin{figure}
    \centering
    \includegraphics[width=0.70\linewidth]{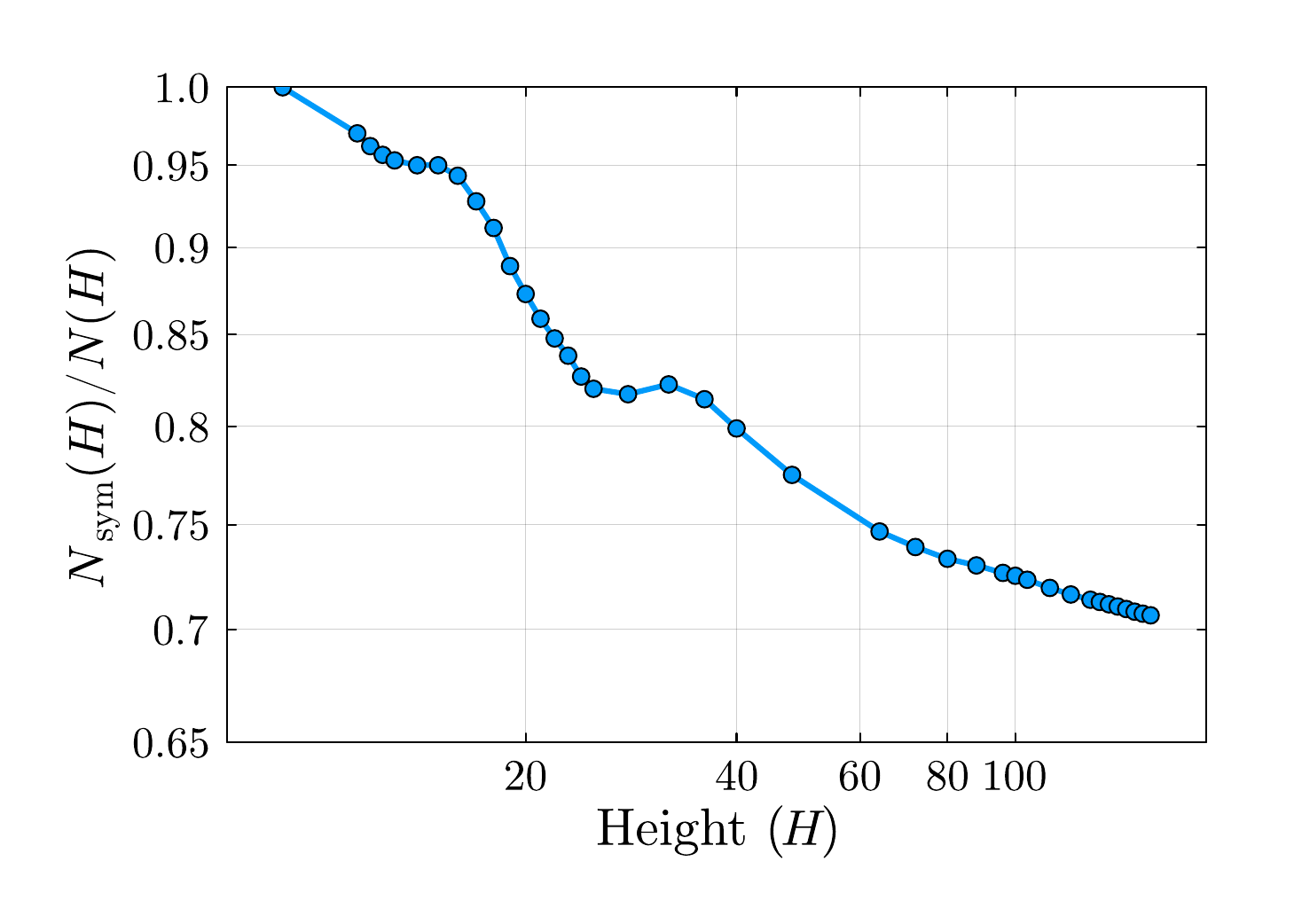}
    \caption{
    Observed symmetric proportion $R(H)$
    among all nontrivial ideal degree-three Prouhet--Tarry--Escott solutions
    of size four in the finite enumeration, counted with the multiset,
    interchange, and admissible-center conventions of the text.  The
    proportion decreases over the displayed range but remains substantial
    there.  
    }
    \label{fig:ratio}
\end{figure}

\FloatBarrier

\section{Alternative Counting Geometry: Finite-Squared-Sum Enumeration}
\label{app:finite-sqSum-enumeration}

In the main body of the paper we count centered symmetric configurations by a height-box cutoff.
This appendix records a companion count for a radial cutoff.
The result is not used in the proof of the height-box asymptotic, but it is useful for comparison: the logarithmic growth again comes from the second moment of the sum-of-two-squares representation function, while the leading constant changes because the cutoff geometry has changed.

We use the following normalization.  For a symmetric centered configuration with center $c$, define its one-side centered squared sum by
\begin{align}
Q(A;c):=\sum_{a\in A}(a-c)^2.
\end{align}
Since the quadratic Prouhet--Tarry--Escott relation gives
\begin{align}
\sum_{a\in A}(a-c)^2=\sum_{b\in B}(b-c)^2,
\end{align}
this is also equal to $Q(B;c)$.  We count centered configurations satisfying
\begin{align}
Q(A;c)=Q(B;c)\le U.
\end{align}
Thus \(U\in\mathbb Z_{>0}\) is a radial cutoff on either one of the two
centered multisets.

A centered configuration has the form
\begin{align}
A_0=\{\pm x,\pm y\},\qquad B_0=\{\pm u,\pm v\},
\qquad
x,y,u,v\in\mathbb Z_{\geq 0}
\ \text{or}\
\mathbb Z_{\geq 0}+\frac{1}{2},
\end{align}
and the common quadratic condition is
\begin{align}
x^2+y^2=u^2+v^2.
\end{align}
If
\begin{align}
p=x^2+y^2=u^2+v^2,
\end{align}
then
\begin{align}
Q(A_0;0)=Q(B_0;0)=2p
\qquad\text{and}\qquad
p\in \frac{1}{2}\mathbb Z.
\end{align}
Hence the radial cutoff \(Q\leq U\) is equivalent to
\begin{align}
p \leq \frac U2 .
\end{align}
The solutions can be divided into two classes,
\(p=n\in\mathbb Z\) and \(p=m\in\mathbb Z+\frac12\), corresponding to the
integer-centered class and the half-integer-centered class, respectively.

Geometrically, the height-box count in Section~\ref{sec:counting} samples the square
\begin{align}
0\le x,y\le H,
\end{align}
whereas the present count samples the quarter disk
\begin{align}
x^2+y^2\le \frac U2.
\end{align}
This is the reason the leading constant below does not contain the factor $\log 2$ appearing in the height-box asymptotic.

\subsection{The Integer-centered Class}

For $n\ge 0$, let
\begin{align}
s(n)
=
\#\bigl\{\{x,y\} : 0\le x\le y,\ x,y\in\mathbb Z,\ x^2+y^2=n\bigr\}.
\end{align}
Thus $s(n)$ counts unordered nonnegative representations of $n$ as a sum of
two squares. In particular, $s(0)=1$.

Let $C_{\mathbb Z}^{\rm rad}(U)$ denote the number of nontrivial centered symmetric configurations in the integer-centered class satisfying $Q\le U$, counted with unordered multiset conventions and with $(A,B)$ identified with $(B,A)$.
Then
\begin{align}
C_{\mathbb Z}^{\rm rad}(U)
=
\frac12
\sum_{n \leq \lfloor U/2\rfloor}
s(n)\bigl(s(n)-1\bigr).
\end{align}

For $n\geq 1$, the relation between $s(n)$ and $r_2(n)$ is
\begin{equation}
s(n)
=
\frac{1}{8} r_2(n)
+
\frac{1}{2}\epsilon(n),
\label{eq:appendix-r2-to-s}
\end{equation}
where
\begin{equation}
\epsilon(n)
=
\begin{cases}
1, & \text{if } n=t^2 \text{ or } n=2t^2
\text{ for some } t\in\mathbb Z_{\geq 1},\\[1mm]
0, & \text{otherwise}.
\end{cases}
\label{eq:appendix-epsilon-def}
\end{equation}

Indeed, a generic representation $0<x<y$ gives eight ordered signed
representations, while the exceptional cases $x=0$ and $x=y$ correspond
respectively to $n=t^2$ and $n=2t^2$.

We recall the standard estimates
\begin{equation}
\sum_{n\le X} r_2(n)=\pi X+o(X),
\label{eq:appendix-r2-first}
\end{equation}
and
\begin{equation}
\sum_{n\le X} r_2(n)^2
=
4X\log X+4\alpha_{\rm BC}X+O(X^{2/3}),
\label{eq:appendix-r2-second}
\end{equation}
following Borwein--Choi~\cite{BorweinChoi2003}.
The second moment \eqref{eq:appendix-r2-second} is the source of the
logarithmic enhancement.
The exceptional contribution in \eqref{eq:appendix-r2-to-s} is of a lower order.
Indeed, from
\begin{align}
r_2(n)=4\sum_{d\mid n}\chi_4(d) \leq 4\sum_{d | n}1 ,
\end{align}
where \(\chi_4\) is the nonprincipal Dirichlet character modulo \(4\), and the standard divisor bound, for every $\varepsilon>0$,
\begin{align}
r_2(t^2)\ll_\varepsilon t^\varepsilon,
\qquad
r_2(2t^2)=r_2(t^2)\ll_\varepsilon t^\varepsilon.
\end{align}
Consequently,
\begin{align}
\sum_{t\le \sqrt X} r_2(t^2)
+
\sum_{t\le \sqrt{X/2}} r_2(2t^2)
\ll_\varepsilon X^{1/2+\varepsilon}
=o(X)
\end{align}
after choosing $\varepsilon<1/2$.

It follows from \eqref{eq:appendix-r2-to-s}--\eqref{eq:appendix-r2-second} that
\begin{align}
\sum_{n\le X}s(n)
=
\frac{\pi}{8}X+o(X),
\end{align}
and
\begin{align}
\sum_{n\le X}s(n)^2
=
\frac{1}{16}X\log X+\frac{\alpha_{\rm BC}}{16}X+o(X).
\end{align}
Therefore
\begin{equation}
\frac12\sum_{n\le X}s(n)\bigl(s(n)-1\bigr)
=
\frac{1}{32}X\log X
+
\left(
\frac{\alpha_{\rm BC}}{32}
-\frac{\pi}{16}
\right)X
+o(X).
\label{eq:appendix-integer-X}
\end{equation}
Taking $X=\lfloor U/2\rfloor$, we obtain
\begin{equation}
C_{\mathbb Z}^{\rm rad}(U)
=
\frac{1}{64}U\log U
+
\left(
\frac{\alpha_{\rm BC}}{64}
-\frac{\pi}{32}
-\frac{\log 2}{64}
\right)U
+o(U).
\label{eq:appendix-integer-radial}
\end{equation}

\subsection{The Half-integer-centered Class}

We now treat the half-integer-centered class.
For $m\in\mathbb Z+\frac12$, define
\begin{align}
s^{\rm half}(m)
=
\#\bigl\{\{x,y\}:x,y\in\mathbb Z_{\ge0}+\tfrac12,\ x\le y,\ x^2+y^2=m\bigr\}.
\end{align}
It is convenient to write
\begin{align}
m=n-\frac12,
\qquad n\in\mathbb Z_{\ge1}.
\end{align}
Let
\begin{align}
r_2^{\rm half}(m)
=
\#\left\{(x,y)\in\left(\mathbb Z+\frac12\right)^2:x^2+y^2=m\right\},
\end{align}
where order and signs are counted.  After multiplying the coordinates by $2$,
one obtains
\begin{align}
r_2^{\rm half}\!\left(n-\frac12\right)
=
r_2(4n-2).
\end{align}
Since $4n-2=2(2n-1)$ and $r_2(2k)=r_2(k)$ for every positive integer
$k$, this becomes
\begin{equation}
r_2^{\rm half}\!\left(n-\frac12\right)
=
r_2(2n-1).
\label{eq:appendix-half-r2}
\end{equation}

As in the integer-centered case, a generic half-integer representation gives
eight ordered signed representations.  The only exceptional case is the
diagonal $x=y$, which contributes only $O(X^{1/2})$ values up to
$m\le X$.  Thus the same boundary analysis gives
\begin{align}
\sum_{n\le X}
s^{\rm half}\!\left(n-\frac12\right)
=
\frac18\sum_{n\le X}r_2(2n-1)+o(X).
\end{align}
Using $r_2(2k)=r_2(k)$, we have
\begin{align}
\sum_{n\le X} r_2(2n-1)
=
\sum_{n\le 2X-1}r_2(n)-\sum_{n\le X-1}r_2(n),
\end{align}
and hence, by \eqref{eq:appendix-r2-first},
\begin{align}
\sum_{n\le X}
s^{\rm half}\!\left(n-\frac12\right)
=
\frac{\pi}{8}X+o(X).
\end{align}
Similarly,
\begin{align}
\sum_{n\le X} r_2(2n-1)^2
=
\sum_{n\le 2X-1}r_2(n)^2-\sum_{n\le X-1}r_2(n)^2.
\end{align}
Using \eqref{eq:appendix-r2-second}, this gives
\begin{align}
\sum_{n\le X} r_2(2n-1)^2
=
4X\log X+
\bigl(4\alpha_{\rm BC}+8\log 2\bigr)X
+o(X).
\end{align}
The diagonal exceptional contribution is again $o(X)$, since
\begin{align}
\sum_{k\le \sqrt X} r_2((2k-1)^2)\ll_\varepsilon X^{1/2+\varepsilon}.
\end{align}
Therefore
\begin{align}
\sum_{n\le X}
s^{\rm half}\!\left(n-\frac12\right)^2
=
\frac{1}{16}X\log X
+
\frac{\alpha_{\rm BC}+2\log 2}{16}X
+o(X).
\end{align}
Consequently,
\begin{equation}
\frac12
\sum_{n\le X}
s^{\rm half}\!\left(n-\frac12\right)
\left(
s^{\rm half}\!\left(n-\frac12\right)-1
\right)
=
\frac{1}{32}X\log X
+
\left(
\frac{\alpha_{\rm BC}}{32}
-\frac{\pi}{16}
+\frac{\log 2}{16}
\right)X
+o(X).
\label{eq:appendix-half-X}
\end{equation}

Let $C_{\mathbb Z+\frac12}^{\rm rad}(U)$ denote the corresponding
nontrivial centered count in the half-integer-centered class.  Since
$m=n-\frac12\le U/2$, replacing the upper limit by $U/2$ changes the
asymptotic by $o(U)$.  From \eqref{eq:appendix-half-X}, we obtain
\begin{equation}
C_{\mathbb Z+\frac12}^{\rm rad}(U)
=
\frac{1}{64}U\log U
+
\left(
\frac{\alpha_{\rm BC}}{64}
-\frac{\pi}{32}
+\frac{\log 2}{64}
\right)U
+o(U).
\label{eq:appendix-half-radial}
\end{equation}

\subsection{The Radial Centered Count}

Adding \eqref{eq:appendix-integer-radial} and
\eqref{eq:appendix-half-radial}, the $\log 2$-terms in the linear
coefficient cancel.  Hence the full centered radial count satisfies
\begin{equation}
C_{\rm sym}^{\rm rad}(U)
=
C_{\mathbb Z}^{\rm rad}(U)
+
C_{\mathbb Z+\frac12}^{\rm rad}(U)
=
\frac{1}{32}U\log U
+
\left(
\frac{\alpha_{\rm BC}}{32}
-\frac{\pi}{16}
\right)U
+o(U).
\label{eq:num_sym_cen_rad}
\end{equation}
This should be compared with the height-box asymptotic
\begin{align}
C_{\rm sym}(H)
=
\frac{2\log 2}{\pi^2}H^2\log H+O(H^2).
\end{align}
Both counts are governed by the same arithmetic second moment of $r_2(n)$,
but the geometric averaging is different.  The radial cutoff orders
representations directly by the value of $x^2+y^2$, while the height-box
cutoff includes the full square $0\le x,y\le H$, whose outer annular region
contributes the additional factor $\log 2$ in the box constant.

Since the asymptotic expansion gives the subleading term for the symmetric
radial count, we compare \eqref{eq:num_sym_cen_rad} with the numerical
enumeration in Fig.~\ref{fig:SolutionsCounts_sqSum}.  As in
Fig.~\ref{fig:SolutionsCounts}, the blue and orange markers show the
symmetric and non-symmetric centered counts with cutoff $U$.  The gray dashed
curve corresponds to
\begin{equation}
\frac{1}{32}U\log U+
\left(\frac{\alpha_{\rm BC}}{32}-\frac{\pi}{16}\right)U,
\label{eq:radial-fit-curve}
\end{equation}
and gives a good description of the symmetric centered count at the displayed
values of $U$.\footnote{Numerically, the difference between the symmetric
count and \eqref{eq:radial-fit-curve} appears to be $o(U)$ over the computed
range.}

\begin{figure}
    \centering
    \includegraphics[width=0.70\linewidth]{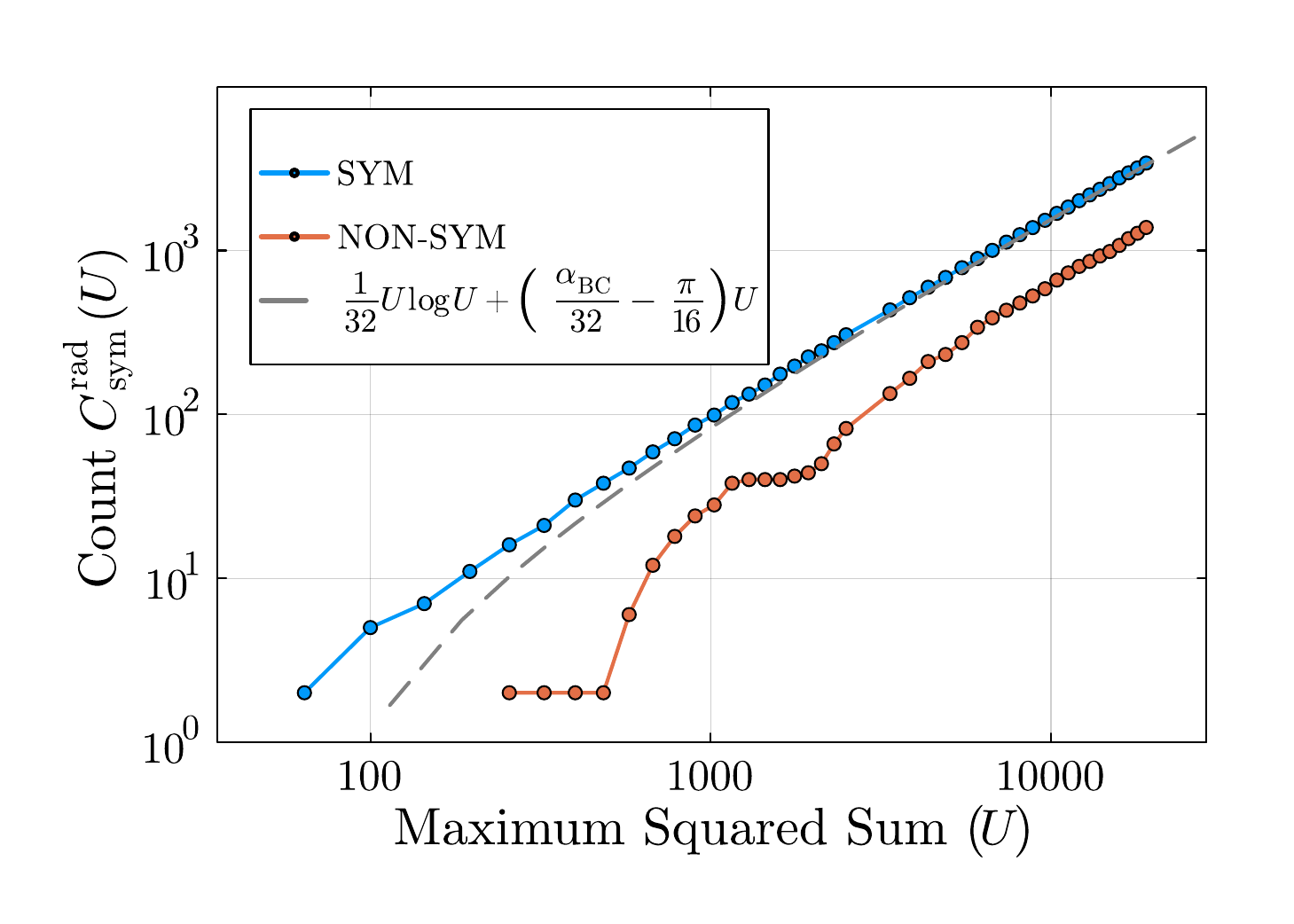}
    \caption{
    Centered nontrivial symmetric solutions ordered by the one-side centered
    squared-sum cutoff $U$.  The dashed curve corresponds to
    \eqref{eq:num_sym_cen_rad}.
    }
    \label{fig:SolutionsCounts_sqSum}
\end{figure}

\FloatBarrier

\newpage
\section*{Emails and Affiliations}

\begingroup

\setlength{\parindent}{0pt}
\setlength{\parskip}{0.8em}

\textbf{Yu-Dai Tsai}\\
\href{mailto:y.tsai@sheffield.ac.uk}{y.tsai@sheffield.ac.uk}\\
\href{mailto:yu-dai.tsai@manchester.ac.uk}{yu-dai.tsai@manchester.ac.uk}\\
\href{mailto:yudaitsai.academic@gmail.com}{yudaitsai.academic@gmail.com}\\
The University of Sheffield, Sheffield S3 7RH, UK\\
The University of Manchester, Manchester M13 9PL, UK\\
Los Alamos National Laboratory (LANL), Los Alamos, NM 87545, USA

\textbf{Junseok Lee}\\
\href{mailto:lee.junseok.p4@dc.tohoku.ac.jp}{lee.junseok.p4@dc.tohoku.ac.jp}\\
Department of Physics, Tohoku University, Sendai, Miyagi 980-8578, Japan

\textbf{Fuminobu Takahashi}\\
\href{mailto:fumi@tohoku.ac.jp}{fumi@tohoku.ac.jp}\\
Department of Physics, Tohoku University, Sendai, Miyagi 980-8578, Japan\\
Kavli IPMU (WPI), UTIAS, University of Tokyo, Kashiwa 277-8583, Japan

\endgroup

\end{document}